\documentclass[12pt,twoside]{article}  
\usepackage{amssymb,amsfonts,amsmath,amsbsy,theorem,amscd,dina42e}
\newcommand{\ol}[1]{\overline{#1}}

\numberwithin{equation}{section}

\newcommand{\R}{\ensuremath{\mathbb{R}}}

\newcommand{\Rd}{\ensuremath{\mathbb{R}^d}}
\newcommand{\Rn}{\ensuremath{\mathbb{R}^n}}
\newcommand{\dm}{d}

\newcommand{\N}{\ensuremath{\mathbb{N}}}

\newcommand{\sd}{\, d}

\newcommand{\eps}{\ensuremath{\varepsilon}}
\newcommand{\weight}[1]{\langle #1\rangle}

\newcommand{\Div}{\operatorname{div}}

\newcommand{\Hzero}{H^{-1}_{(0)}}
\newcommand{\Hone}{H^1_{(0)}}

\newcommand{\A}{\mathcal{A}}
\newcommand{\B}{\mathcal{B}}

\newcommand{\ve}{\mathbf{v}}
\newcommand{\we}{\mathbf{w}}
\newcommand{\ue}{\mathbf{u}}
\newcommand{\vc}[1]{\mathbf{#1}}
\newcommand{\tn}[1]{\mathbb{#1}}

\newtheorem{thm}{THEOREM}[section]
\newtheorem{cor}[thm]{Corollary}
\newtheorem{lem}[thm]{Lemma}

\newtheorem{prop}[thm]{Proposition}

\newtheorem{claim*}{Claim}
\newtheorem{assumption}[thm]{Assumption}
\theorembodyfont{\rmfamily}
\newtheorem{rem}[thm]{Remark}

\newenvironment{proof*}[1]{{\bf Proof
#1:}}{\hspace*{\fill}\rule{1.2ex}{1.2ex}\\ } 
\newenvironment{proof}{{\bf
Proof:\,}}{\hspace*{\fill}\rule{1.2ex}{1.2ex}\\ }

\newcommand{\hrho}{\hat{\rho}}
\newcommand{\pot}{\phi}
\newcommand{\Pot}{\Phi}
\newcommand{\normal}{n}

\newcommand{\free}{\text{free}}

\begin{document}
\begin{titlepage}
\title{Strong Well-Posedness of a  Diffuse Interface Model for a Viscous, Quasi-Incompressible Two-Phase Flow}
\author{  Helmut Abels\footnote{Fakult\"at f\"ur Mathematik,  
Universit\"at Regensburg,
93040 Regensburg,
Germany, e-mail: {\sf helmut.abels@mathematik.uni-regensburg.de}
}}
\end{titlepage}
\maketitle
\begin{abstract}
  We study a diffuse interface model for the flow of two viscous
 incompressible Newtonian fluids in a bounded domain. The
 fluids are assumed to be macroscopically immiscible, but a partial mixing in
 a small interfacial region is assumed in the model. Moreover, diffusion of
 both components is taken into account. 
In contrast to previous
 works, we study a model for the general case that the fluids have different densities due to Lowengrub and Truskinovski~\cite{LowengrubQuasiIncompressible}. This
 leads to an inhomogeneous Navier-Stokes system coupled to a Cahn-Hilliard
 system, where the density of the mixture depends on the concentration, the velocity field is
 no longer divergence free, and the pressure enters the equation for the
 chemical potential.
We prove
 existence of unique strong solutions for the non-stationary system for sufficiently small times.
\end{abstract}

\noindent{\bf AMS-Classification:} 
Primary: 76T99, 
Secondary: 35Q30, 
35Q35, 
76D27, 
76D03, 
76D05, 
76D45 

\section{Introduction}

\label{sec:IntroLT} 

In this article we consider a so-called diffuse interface model for two
viscous, incompressible Newtonian fluids of different densities. In the model a
partial mixing of the macroscopically immiscible fluids is considered and
diffusion effects are taken into account. 
Such models have been successfully used to describe flows of two or more macroscopically fluids beyond the occurrence of topological singularities of the separating interface (e.g. coalescence or formation of drops).  We refer to Anderson and McFadden~\cite{DiffIntModels} for a review on that topic.

The model which we are considering leads  
to the system
\begin{alignat}{2}\label{eq:FirstLT1}
  \rho \partial_t \ve + \rho \ve\cdot \nabla \ve -\Div \tn{S}(c,\tn{D}\ve)+ \nabla p&= -\Div (a(c)\nabla c\otimes \nabla c)&\quad & \text{in}\ Q_T,\\
  \partial_t \rho +\Div (\rho \ve)&= 0 && \text{in}\ Q_T,\\
  \rho \partial_t c+ \rho \ve\cdot \nabla c &= \Div(m(c) \nabla \mu)&& \text{in}\ Q_T,
\end{alignat}
\begin{equation}
\label{eq:FirstLT4}
 \rho \mu = -\rho^{-1}\frac{\partial\rho}{\partial c}\left(p + \Phi(c)+ \frac{a(c)|\nabla c|^2}2\right) + \pot(c) -a(c)^\frac12\Div(a(c)^\frac12\nabla c)
\end{equation}
where $Q_T=\Omega \times (0,T)$ and $\Omega\subset
\R^d$, $d=2,3$, is a bounded domain with $C^3$-boundary.
Here $\ve$ and $\rho=\hrho(c)$ are the (mean) velocity  and the
density of the mixture of the two fluids, $p$ is the pressure,
 $c$ is
the difference of the mass concentrations of the two fluids, and $\mu$ is the chemical
potential associated to $c$. Moreover, 
\begin{equation}\label{eq:DefnS}
  \tn{S}(c,\tn{D}\ve) = 2\nu(c) \tn{D}\ve + \eta(c) \Div \ve\, \tn{I},
\end{equation}
where $\tn{S}(c,\tn{D}\ve)$ is the stress tensor, $\tn{D}\ve= \frac12(\nabla \ve + \nabla \ve^T)$,
$\nu(c),\eta(c)>0$ are two viscosity coefficients, and $m(c)>0$ is a mobility coefficient. Furthermore,
$\Phi(c)$ is the homogeneous free energy density for the mixture and $\phi(c)=\Phi'(c)$.

This is a variant of the model proposed
by Lowengrub and Truskinovski~\cite{LowengrubQuasiIncompressible} for an interfacial energy of the form
\begin{equation*}
    E_{\free}(c) = \int_\Omega \Pot(c) \sd x + \int_\Omega a(c)\frac{|\nabla c|^2}2\sd x,
\end{equation*}
where the choice $a(c)=\rho(c)$ was proposed in \cite{LowengrubQuasiIncompressible}. A
derivation of the latter system can also be found in
\cite[Chapter~II]{Habilitation} (in an even more general form). We note that the term
$a(c)\nabla c\otimes \nabla c$ comes from an extra contribution
to the stress tensor, which models capillary forces in an interfacial region. The model is a generalization of a
well-known diffuse interface model in the case of matched densities which
corresponds to the case $\hrho(c)\equiv const.$,
cf. e.g. Gurtin et al.~\cite{GurtinTwoPhase}.

The system is equivalent to
\begin{alignat}{2}\label{eq:LT1}
  \rho \partial_t \ve + \rho \ve\cdot \nabla \ve -\Div \tn{S}(c,\tn{D}\ve)
  + \rho \nabla g_0&= \rho \mu_0 \nabla c&\quad& \text{in}\ Q_T,\\
  \label{eq:LT2}
  \partial_t \rho +\Div (\rho \ve)&= 0 && \text{in}\ Q_T,\\
  \label{eq:LT3}
  \rho \partial_t c+ \rho \ve\cdot \nabla c &= \Div (m(c)\nabla \mu_0)&&
  \text{in}\ Q_T,
\\  \label{eq:LT4}
  \rho \mu_0 + \rho^2\bar{p} =\beta \rho^2 g_0 -a(c)^\frac12 \Div (a(c)^\frac12\nabla c) &+ \pot(c)&& \text{in}\ Q_T,
\end{alignat}
together with
\begin{equation}\label{eq:LT42}
  \int_\Omega \mu_0(t) \sd x=\int_\Omega g_0(t) \sd x= 0 \quad \text{for all}\ t\in
  (0,T), 
\end{equation}
where  $\mu_0$ is the mean-value free part of $\mu$, and
$\bar{p}$ is a constant (depending on time), which is related to the mean values
of the pressure and the chemical potential. Here $p$ and $g$ are related by
\begin{equation*}
  g= \frac{\Pot(c)}\rho + \frac{a(c)|\nabla c|^2}{2\rho} + \frac{p}\rho - \bar{\mu} c, 
\end{equation*}
and $\mu = \mu_0+ \bar{\mu}$, $\bar{\mu}= \frac1{|\Omega|} \int_\Omega \mu \sd x$, $g = g_0+ \bar{g}$, $\bar{g}= \frac1{|\Omega|} \int_\Omega g \sd x$. Details on this equivalence can be found in \cite[Section~3]{LTModel}.  Here it is assumed that the fluids mix with zero excess volume,
cf. \cite{LowengrubQuasiIncompressible}. This implies 
\begin{equation*}
  \frac1{\hrho(c)} = \frac{1+c}{2 \bar{\rho}_1} +\frac{1-c}{2\bar{\rho}_2}, 
\end{equation*}
where $\bar{\rho}_j$ is the specific density of fluid $j=1,2$. Hence $\hrho(c)$ is of the form
\begin{equation}
  \label{eq:rho}
\hrho(c)=
\frac1{\alpha + \beta c}
\end{equation}
for  $\alpha>0$ and $|\beta|<\alpha$. (More precisely, $\beta= \frac1{2\bar{\rho}_1}-\frac1{2\bar{\rho}_2}, \alpha=
\frac1{2\bar{\rho}_2}+\frac1{2\bar{\rho}_1}$.)

We close the system by adding the boundary and initial conditions
\begin{alignat}{2}\label{eq:LT5}
  \vc{n}\cdot\ve|_{\partial\Omega}= \left.(\vc{n}\cdot \tn{S}(c,\tn{D}\ve))_\tau+\gamma(c) \ve_\tau\right|_{\partial\Omega}
  &= 0 &\quad& \text{on}\ S_T,\\\label{eq:LT6}
  \partial_n c|_{\partial\Omega}=\partial_n \mu_0 |_{\partial\Omega}
  &= 0 &\quad& \text{on}\ S_T,\\\label{eq:LT7}
  (\ve,c)|_{t=0} &=(\ve_0,c_0) && \text{in}\ \Omega,
\end{alignat}
where $S_T= \partial\Omega \times (0,T)$. I.e., we assume that $\ve$ satisfies Navier boundary conditions with some friction parameter $\gamma\colon \R\to [0,\infty)$  and assume Neumann boundary conditions for $c$ and $\mu$. 

In
the  case of matched densities, i.e., $\hrho\equiv const.$, $\beta=0$, resp., results on existence of weak solutions and well-posedness  were obtained by Starovoitov~\cite{StarovoitovModelH}, Boyer~\cite{BoyerModelH}, Liu
and Shen~\cite{LiuShenModelH}, and the author~\cite{ModelH}. The long time dynamics was studied by Gal and Grasselli~\cite{GalGrasselliModelH2D,GalGrasselliDCDS,GalGrasselliTrajAttr}, Zhao et al.~\cite{ZhaoEtAl} and in \cite{LongtimeModelH,ModelH}.
Moreover, in
\cite{BoyerNonMatched} Boyer considered a different diffuse interface model for fluids with
non-matched densities. He proved existence of strong solutions, locally in
time, and existence of global weak solutions if the densities of the fluids
are sufficiently close.
In the case of general densities, existence of weak solutions of a slightly modified system was shown in \cite{LTModel}, where the  case of a free energy of the form
\begin{equation*}
    E_{\free}(c) = \int_\Omega \Pot(c) \sd x + \int_\Omega a(c)\frac{|\nabla c|^q}q\sd x
\end{equation*}
with $q>d$ is considered. This is the only analytic result for the model \eqref{eq:FirstLT1}-\eqref{eq:FirstLT4} so far known to the author. To the authors knowledge there are no numerical studies of this model in the case $\beta\neq 0$. A simplified model was used by Lee et al.~\cite{LeeLowengrub1,LeeLowengrub2} in numerical simulations.
 Moreover, A. and Feireisl~\cite{CompressibleNSCH}
constructed weak solutions globally in time for a corresponding diffuse
interface model for compressible fluids.

In the following, we will only
consider the case $\rho(c)\not\equiv const.$, i.e., $\beta\neq 0$ and will restrict ourselves for simplicity to the case $a(c)=m(c)=1$. Moreover, we do the following assumption:
\begin{assumption}\label{assump:3}
  Let $\Omega\subset \R^d$, $d=2,3$, be a bounded domain with $C^3$-boundary,
  let $\eps>0$, $\alpha>0$, and $\beta \neq 0$ such that
  $|\beta|<\alpha$. Moreover, we assume $\nu,\eta,\gamma\in C^2(\R)$ such that
  $\inf_{s\in\R}\nu(s),\inf_{s\in\R}\eta(s)>0$, $\inf_{s\in \R} \gamma(s)\geq 0$, $\Phi\in C^3(\R)$ and $\hrho\in C^3(\R)$ such that
  $\hrho(s) = \frac1{\alpha +\beta s}$ if $s\in [-1-\eps_0,1+\eps_0]$ for some $\eps_0$ and $\hrho(s)>0$ for
  all $s\in\R$.  Finally, if $\inf_{s\in\R} \gamma(s)=0$, then we assume that $\Omega$ has no axis of symmetry, cf. Appendix for details.
\end{assumption}
Now our main result on short time existence of strong solutions is:
\begin{thm}\label{thm:main} 
  Let $\ve_0\in H^1_n(\Omega), c_0\in H^2(\Omega)$ with $|c_0(x)|\leq 1$ almost everywhere and $\partial_n c_0|_{\partial\Omega}=0$, $d=2,3$, 
and let
  Assumption~\ref{assump:3} hold. Then there is some $T>0$ such that there is a unique solution $\ve\in H^1(0,T;L^2_\sigma(\Omega))\cap L^2(0,T;H^2(\Omega)^d)), c\in H^2(0,T;H^{-1}_{(0)}(\Omega))\cap L^2(0,T;H^3(\Omega))$ solving \eqref{eq:LT1}-\eqref{eq:LT4},\eqref{eq:LT5}-\eqref{eq:LT7}.
\end{thm}
Precise definitions of the function spaces are given in Section~\ref{sec:preliminaries} below.

The theorem is proved by linearizing the system suitably, proving that the linearized operator defines an isomorphism between certain $L^2$-Sobolev spaces, and applying a contraction mapping argument. To apply this general strategy it is essential to reformulate the system \eqref{eq:LT1}-\eqref{eq:LT4}. To this end, we eliminate $\mu_0$ and $g_0$ first from the system. Then the principal part of the linearized system (around $(\ve_0,c_0)$) is
\begin{eqnarray}\label{eq:lin1}
      \partial_t \ve - \Div \widetilde{\tn{S}}(c_0,\tn{D}\ve) + \frac{\eps}{\beta\alpha}\nabla \Div (\rho^{-4}_0 \nabla c')&=& \vc{f}_1\qquad \text{in}\ Q_T,\\\label{eq:lin2}
     \partial_t c' - \beta^{-1} \Div \ve&=& f_2 \qquad \text{in}\ Q_T,
\end{eqnarray}
where $c'\approx \rho c$ and $\rho_0=\hrho(c_0)$, cf. Section~\ref{sec:Strong} below. One of the essential steps in the proof of the main result is the analysis of this linearized system. 
To this end we split $\ve$ in a divergence free part $\we=P_\sigma \ve$, a gradient part $\nabla G(\Div \ve)$, which is determined uniquely by $g=\Div v$, and a lower order part, cf. Section~\ref{sec:LinearPart} for details. A crucial observation is that $c'$, which is related to $\Div \ve$ via \eqref{eq:lin2}, solves a kind of damped wave/plate equation. More precisely, $c'$ solves an equation of the form
\begin{equation}\label{eq:dampedWave}
  \partial_t^2 c' -\Delta (a(c_0)\partial_t c' )+  \frac\eps{\alpha\beta^2} \Delta\Div (\rho^{-4}_0\nabla c')= f 
\end{equation}
up to to lower order terms for some $a(c_0)>0$. In order to solve this equation we will apply the abstract result of Chen and Triggiani~\cite{ChenTriggiani}. -- We note that the same kind of linearized system arises in Kotschote~\cite{KotschoteKorteweg}, where existence of strong solutions locally in time is proved for a compressible Navier-Stokes-Korteweg system.
\begin{rem}
  It is interesting to compare \eqref{eq:lin1}-\eqref{eq:lin2} to the linearized system of the Model H for the case of matched densities, i.e, \eqref{eq:FirstLT1}-\eqref{eq:FirstLT4} in the case when $\beta=0$ and thus $\rho(c)\equiv const.$ Then the pressure $p$ is no longer part of \eqref{eq:FirstLT4}, $\Div \ve =0$, $p$ can no longer be eliminated from the system, and  the principal part of the linearized system  is 
\begin{alignat*}{2}
      \partial_t \ve - \Div \widetilde{\tn{S}}(c_0,\tn{D}\ve) + \nabla p&= \vc{f}_1&\qquad& \text{in}\ Q_T,\\
\Div \ve &= 0 &&  \text{in}\ Q_T,\\
     \partial_t c +\Div (m(c_0)\nabla \Delta c)&= f_2 &\qquad& \text{in}\ Q_T.
\end{alignat*}
Hence the linearized system is very different. In particular, the principal part for $c$ is given by a fourth order diffusion equation with $\sigma(-\Div (m(c)\nabla \Delta c))\subset (-\infty,0]$ for a suitable realization. While the corresponding operator $A_1$ to \eqref{eq:dampedWave} (after reduction to a first order system, cf. Section~\ref{sec:LinearPart}) still generates an analytic semigroup, but the spectral angle $\delta<\frac{\pi}2$ can be arbitrarily close to $\frac{\pi}2$ in certain situations, cf. Remark~\ref{rem:Spectrum} below. 
Moreover, note that  the Cahn-Hilliard part is decoupled from the  Navier-Stokes part on the level of the principal part of the linearized system in the case $\beta=0$, which is no longer the case if $\beta \neq 0$. -- We hope that the insight on the analytic structure of the system~\eqref{eq:LT1}-\eqref{eq:LT4} will help to create stable numerical algorithms, which are not available so far to the best of the author's knowledge.
\end{rem}

The structure of the article is as follows: In Section~\ref{sec:preliminaries} we summarize some notation and preliminary results. The main part of the article consists of Sections~\ref{sec:Strong} and~\ref{sec:LinearPart}. First, in Section~\ref{sec:Strong} the system is linearized and the contraction mapping principle is applied on the basis of the  well-posedness result Theorem~\ref{thm:LinearPart} for the linearized system. Afterwards, in Section~\ref{sec:LinearPart} this result is proved.

\section{Preliminaries}\label{sec:preliminaries}

\noindent {\bf Notation:}
Let us fix some notation first. For $a,b\in \Rd$ let $a\otimes b\in \R^{d\times d}$ be defined by $(a\otimes b)_{i,j}= a_ib_j$. Moreover, for $A,B\in \R^{d\times d}$ let $A:B= \operatorname{tr} (A^TB)= \sum_{i,j=1}^d a_{i,j}b_{i,j}$. In the following $\vc{n}$ will denote the exterior normal at the boundary of a sufficiently smooth domain $\Omega\subset\R^d$. Furthermore, $\vc{f}_n:=\vc{n}\cdot \vc{f}$ and $\vc{f}_\tau:=(I-\vc{n}\otimes \vc{n})f=\vc{f}-\vc{f}_n\vc{n}$ denote the normal and tangential component of a vector field $\vc{f}\colon \partial\Omega\to \R^d$, respectively.
Furthermore $\partial_n := \vc{n}\cdot \nabla $, $\nabla_{\tau}:= (I-\vc{n}\otimes \vc{n})\nabla $, and  $\partial_{\tau_j}:= e_j\cdot (I-\vc{n}\otimes \vc{n})\nabla $, $j=1,\ldots, d$, where $e_j$ denotes the $j$-th canonical unit vector in $\R^d$. If $\vc{v}\in C^1(\Omega)^d$, then $\nabla \vc{v}=D\vc{v}=(\partial_{x_j}v_i)_{i,j=1}^d$ denotes its Jacobian. Moreover, if $A=(a_{ij})_{i,j=1}^d\colon \Omega\to \R^{d\times d}$ is differentiable, then $\Div A(x):= (\sum_{j=1}^d\partial_{x_j} a_{ij}(x))_{i=1}^d$ for all $x\in\Omega$.

If $X$ is a Banach space and $X'$ is its dual, then
\begin{equation*}
  \weight{f,g}\equiv\weight{f,g}_{X',X} = f(g), \qquad f\in X',g\in X,
\end{equation*}
denotes the duality product. The inner product on a Hilbert space $H$ is
denoted by $(.,.)_H$. Moreover, we use the abbreviation $(.,.)_{M}=
(.,.)_{L^2(M)}$. 

\medskip

\noindent {\bf Function spaces:}
If $M\subseteq \R^d$ is measurable, 
$L^q(M)$, $1\leq q \leq \infty$ denotes the usual Lebesgue-space and $\|.\|_q$
its norm.
Moreover, $L^q(M;X)$ denotes its vector-valued variant of strongly measurable
$q$-integrable functions/essentially bounded functions, where $X$ is a Banach
space. If $M=(a,b)$, we write for simplicity $L^q(a,b;X)$ and $L^q(a,b)$.
 
Let $\Omega \subset \Rd$ be a domain. 
Then $W^m_q(\Omega)$, $m\in \N_0$, $1\leq q\leq \infty$, denotes the usual $L^q$-Sobolev space, 
$W^m_{q,0}(\Omega)$  the closure of $C^\infty_0(\Omega)$ in $W^m_q(\Omega)$,
$W^{-m}_q(\Omega)= (W^m_{q',0}(\Omega))'$, and $W^{-m}_{q,0}(\Omega)= (W^m_{q'}(\Omega))'$. 
The $L^2$-Bessel potential spaces are denoted by
$H^s(\Omega)$, $s\in\R$, which are defined by restriction of distributions in
$H^s(\Rd)$ to $\Omega$, cf. Triebel~\cite[Section~4.2.1]{Triebel1}. Moreover, $H^s(M)$ denotes the corresponding space on a sufficiently smooth compact manifold $M$. We note
that, if $\Omega\subset \R^\dm$ is a bounded domain with $C^{0,1}$-boundary,
then there is an extension operator $E_\Omega$ which is a bounded linear operator
$E_\Omega\colon W^m_p(\Omega)\to W^m_p(\R^\dm)$, $1\leq p\leq \infty$ for all
$m\in \N$ and $E_\Omega f|_{\Omega} = f$ for all $f\in W^m_p(\Omega)$,
cf. Stein~\cite[Chapter VI, Section 3.2]{Stein:SingInt}. This extension
operator extends to $E_\Omega\colon H^s(\Omega)\to H^s(\Rn)$, which shows that
$H^s(\Omega)$ is a retract of $H^s(\Omega)$. Therefore all results on
interpolation spaces of $H^s(\Rn)$ carry over to $H^s(\Omega)$. We refer to Bergh and L\"ofstr\"om~\cite{Interpolation} for basic results in interpolation theory. In the following $(.,.)_{[\theta]}$ and $(.,.)_{\theta,q}$ will denote the complex and real interpolation function, respectively.
In particular, we note that
\begin{equation}\label{eq:BesselInterpol} 
  (H^{s_0}(\Omega), H^{s_1}(\Omega))_{[\theta]} =  
  (H^{s_0}(\Omega), H^{s_1}(\Omega))_{\theta,2} = H^s(\Omega) 
\end{equation}
for all $\theta \in (0,1)$ where $s= (1-\theta)s_0 + \theta s_1$, $s_0,s_1\in\R$.

Moreover, we define 
\begin{equation*}
H^1_n(\Omega)= \left\{\ue \in H^1(\Omega)^d: \vc{n}\cdot\ue|_{\partial\Omega}=0\right\}.  
\end{equation*}

The usual Besov spaces on a domain or a sufficiently smooth manifold are denoted by $B^{s}_{p,q}(\Omega)$, $B^{s}_{p,q}(M)$, resp., where $s\in\R$, $1\leq p,q\leq \infty$. For the convenience of the reader we recall that $B^s_{2,2}(\R^d)= H^s(\R^d)$ for all $s\in\R$ and
\begin{equation*}
  B^{s+\eps}_{p,\infty}(\Rd)\hookrightarrow B^s_{p,q_1}(\Rd)\hookrightarrow B^s_{p,q_2}(\R^d)\quad \text{for all}\ 1\leq q_1\leq q_2\leq \infty.
\end{equation*}
Moreover, we have the Sobolev type embeddings
  \begin{alignat*}{2}
    B^{s_1}_{p_1,q}(\Rn) &\hookrightarrow B^{s_0}_{p_0,q}(\Rn) &\qquad& \text{if}\ s_1\geq s_0 \ \text{and}\ s_1-\tfrac{n}{p_1}\geq s_0-\tfrac{n}{p_0},\\
    B^{d/p}_{p,1}(\Rn) &\hookrightarrow C^0_b(\Rn)
  \end{alignat*}
  for any $1\leq p,q\leq\infty$.
Finally, due to Hanouzet~\cite[Th\'eor\`eme 3]{HanouzetBesovProd} we have the useful product estimate
\begin{equation}\label{eq:BesovProd}
  \|fg\|_{H^1(\R^d)}\leq C_{p}\|f\|_{B^{d/p}_{p,1}(\R^d)}\|g\|_{H^1(\R^d)}
\end{equation}
for all $f\in B^{\frac{d}p}_{p,1}(\R^d),g\in H^1(\R^d)$ provided that  $2\leq p\leq \infty$, see also \cite[Theorem~6.6]{JohnsenPointwiseMultipliers}. 
All these results carry over to sufficiently smooth domains and $d$-dimensional manifolds.

Let $I=[0,T]$ with $0<T< \infty$ or let $I=[0,\infty)$ and let $X$ be a Banach
space. Then $BUC(I;X)$ is the Banach space of all bounded and uniformly continuous
$f\colon I\to X$ equipped with the supremum norm. The space of all (uniformly) H\"older continuous functions $f\colon I\to X$ of degree $s\in (0,1)$ is denoted by $C^s([0,T];X)$ normed in the standard way. 
Furthermore, we have the useful embedding
\begin{equation*}
  BUC([0,T];X_1)\cap C^{s}([0,T];X_0)\hookrightarrow C^{s(1-\theta)}([0,T];X),
\end{equation*}
where $0<s,\theta< 1$ provided that $\|f\|_{X}\leq C\|f\|_{X_0}^{1-\theta}\|f\|_{X_1}^\theta$ for all $f\in X_0\cap X_1$.

Finally, $f\in W^k_p(0,T;X)$, $1\leq p <\infty$, $k\in\N_0$, if and only if $f,\ldots,
\frac{d^kf}{dt^k}\in L^p(0,T;X)$, where $\frac{d^kf}{dt^k}$ denotes the $k$-th $X$-valued
distributional derivative of $f$.
 Furthermore, we set $H^1(0,T;X)= W^1_2(0,T;X)$ and for $s\in (0,1)$ we define $ H^s(0,T;X)= B^s_{2,2}(0,T;X)$, where $f\in B^s_{2,2}(0,T;X)$ if and only if $f\in L^2(0,T;X)$ and
 \begin{equation*}
   \|f\|_{B^s_{2,2}(0,T;X)}^2= \|f\|_{L^2(0,T;X)}^2 + \int_0^T\int_0^T\frac{\|f(t)-f(\tau)\|_{X}^2}{|t-\tau|^{2s+1}}\sd t\sd \tau <\infty.
 \end{equation*}
In the following we will use that
\begin{eqnarray*}
  \lefteqn{\int_0^T\int_0^T\frac{\|f(t)-f(\tau)\|_{X}^2}{|t-\tau|^{2s+1}}\sd t\sd \tau}\\
&\leq& \int_0^T\int_0^T|t-\tau|^{2(s'-s)-1}\sd t\sd \tau \|f\|_{C^{s'}([0,T];X)}
\leq C_{s',s}T^{2(s'-s)+1} \|f\|_{C^{s'}([0,T];X)}
\end{eqnarray*}
for all $0<s<s'\leq 1$, which implies
\begin{equation}\label{eq:VecHEstim}
  \|f\|_{H^{s}(0,T;X)}\leq C_{s,s'}T^{\frac12}\|f\|_{C^{s'}([0,T];X)}\quad \text{for all}\ f\in C^{s'}([0,T];X)
\end{equation}
provided that $0<s<s'\leq 1$. 
Finally, we set for $s\in (0,1)$
\begin{equation*}
  H^{\frac{s}2,s}(S_T)= L^2(0,T;H^s(\partial\Omega))\cap H^{\frac{s}2}(0,T;L^2(\partial\Omega)),
\end{equation*}
where $S_T=\partial\Omega\times (0,T)$ and $\Omega$ is a bounded domain with $C^1$-boundary.

Now let $X_0,X_1$ be Banach spaces such that $X_1\hookrightarrow X_0$ densely. Then
\begin{equation} 
  \label{eq:BUCEmbedding}
   W^1_p(I;X_0) \cap L^p(I;X_1) \hookrightarrow BUC(I;(X_0,X_1)_{1-\frac1p,p}), \qquad 1\leq p <\infty,
\end{equation}
continuously for $I=[0,T]$, $0<T<\infty$,  and $I=[0,\infty)$, cf. Amann~\cite[Chapter III, Theorem 4.10.2]{Amann}.

In order to solve the linearized system in the following, we will use the following abstract result:
\begin{thm}\label{thm:MaxReg} 
  Let $A\colon \mathcal{D}(A)\subset H\to H$ be a generator of a bounded analytic semigroup on a Hilbert space $H$ and let $1<q<\infty$. Then for every $f\in L^q(0,\infty;H)$ and $u_0 \in (H,\mathcal{D}(A))_{1-\frac1q,q}$ there is a unique $u\colon [0,\infty)\to H$ such that $\frac{du}{dt}, Au \in L^q(0,\infty;H)$ solving 
  \begin{eqnarray*}
    \frac{du}{dt}(t) + Au(t) &=& f(t)\quad\text{for all}\ t>0,\\
    u(0)&=& u_0.
  \end{eqnarray*}
Moreover, there is a constant $C_q>0$ independent of $f$ and $u_0$   such that 
  \begin{equation*}
    \left\|\frac{du}{dt}\right\|_{L^q(0,\infty;H)}+ \|Au\|_{L^q(0,\infty;H)} \leq C_q\left(\|f\|_{L^q(0,\infty;H)}+ \|u_0\|_{(H,\mathcal{D}(A))_{1-\frac1q,q}}\right).
  \end{equation*}
\end{thm}
\begin{proof}
In the case $u_0=0$ the statement is the main result of \cite{DeSimonMaxReg}. The general case can be easily reduced to the case $u_0=0$ by subtracting a suitable extension. The existence of such an extension follows e.g. from \cite[Chapter~III, Theorem 4.10.2]{Amann}.
\end{proof}

\medskip

\noindent {\bf Weak Neumann Laplace equation:}
In the following we assume that $\Omega \subset\Rd$ is a bounded domain with $C^{0,1}$-boundary.
Given $f\in L^1(\Omega)$, we denote by 
$
  m(f) = \frac1{|\Omega|}\int_\Omega f(x) \sd x
$
its mean value. Moreover, for $m\in\R$ we set
\begin{equation*}
  L^q_{(m)}(\Omega):=\{f\in L^q(\Omega):m(f)=m\}, \qquad 1\leq q\leq \infty.  
\end{equation*}
Then
$$
  P_0 f:= f-m(f)= f-\frac1{|\Omega|}\int_\Omega f(x) \sd x
$$
is the orthogonal projection onto $L^2_{(0)}(\Omega)$. 
Furthermore,
we define
\begin{equation*}
 \Hone\equiv\Hone (\Omega)= H^1(\Omega)\cap L^2_{(0)}(\Omega), \qquad (c,d)_{\Hone(\Omega)} := (\nabla c,\nabla d)_{L^2(\Omega)}.  
\end{equation*}
Then $\Hone(\Omega)$ is a Hilbert space due to Poincar\'e's inequality
\begin{equation*}
  \|f-m(f)\|_{L^p(\Omega)} \leq C_p\|\nabla f\|_{L^p(\Omega)},
\end{equation*}
where $1\leq p <\infty$.
Moreover, let $\Hzero\equiv\Hzero(\Omega)= \Hone(\Omega)'$. 
Then the \emph{weak Neumann-Laplace operator} $\Delta_N\colon \Hone(\Omega)\to \Hzero(\Omega)$ is defined by
\begin{equation*}
  -\weight{\Delta_N u,\varphi}_{\Hone,\Hzero} =(\nabla u,\nabla \varphi)\qquad \text{for all}\ \varphi \in \Hone(\Omega).
\end{equation*}
By the Lemma of Lax-Milgram, for every $f\in \Hzero(\Omega)$ there is a unique $u\in \Hone(\Omega)$ such that $-\Delta_N u=f$.
More precisely, $-\Delta_N$ coincides with the Riesz isomorphism $\mathcal{R}\colon \Hone(\Omega)\to \Hzero(\Omega)$ given by
\begin{equation*}
\weight{\mathcal{R} c,d}_{\Hzero,\Hone} = (c,d)_{\Hone}= (\nabla c,\nabla d)_{L^2}, \qquad
c,d\in \Hone(\Omega).
\end{equation*}
We equip $\Hzero(\Omega)$ with the inner product
\begin{equation}\label{eq:HzeroInnerProductId}
  (f,g)_{\Hzero} = (\nabla \Delta_N^{-1} f, \nabla \Delta_N^{-1} g)_{L^2} =(\Delta_N^{-1} f, \Delta_N^{-1} g)_{\Hone}. 
\end{equation}
In particular this implies the useful identity
\begin{equation}\label{eq:DeltaNId}
  ((-\Delta_N)f,g)_{\Hzero} = (f, g)_{L^2}\qquad \text{for all}\ f\in H^1_{(0)}(\Omega),g\in L^2_{(0)}(\Omega).  
\end{equation}
Moreover, we embed $\Hone(\Omega)$ and $L^2_{(0)}(\Omega)$ into $\Hzero(\Omega)$ in the standard way by defining
\begin{equation*}
  \weight{c,\varphi}_{\Hzero,\Hone} = \int_\Omega c(x) \varphi(x) \sd x\qquad \text{for all}\ \varphi\in \Hone(\Omega)
\end{equation*}
for $c\in L^2_{(0)}(\Omega)$.
This implies the useful interpolation inequality
\begin{equation}
  \label{eq:InterpolationInequality}
  \|f\|_{L^2}^2 = - (\nabla \Delta_N^{-1}f,\nabla f)_{L^2}\leq \|f\|_{\Hzero}\|f\|_{\Hone}\quad \text{for all}\ f\in\Hone(\Omega).
\end{equation}
Furthermore, if $\Delta_N u=f$ for some $f\in \Hzero(\Omega)$, then 
\begin{equation}\label{eq:WeakNeumannEstim} 
  \|u\|_{\Hone(\Omega)}\leq\|f\|_{\Hzero(\Omega)}.
\end{equation}
We note that, if $u\in H^1_{(0)}(\Omega)$ solves $\Delta_Nu=f$ for some $f\in  L^q_{(0)}(\Omega)$, $1<q<\infty$, and $\partial\Omega$ is of class $C^2$, then it follows from standard elliptic theory that $u\in W^2_q(\Omega)$, $\Delta u= f$ a.e. in $\Omega$, and $\partial_n u|_{\partial\Omega}=0$ in the sense of traces. 
If additionally $f\in W^1_q(\Omega)$ and $\partial\Omega\in C^3$, then $u\in W^3_q(\Omega)$. Moreover,
\begin{equation}\label{eq:W2qEstim}
  \|u\|_{W^{k+2}_q(\Omega)} \leq C_q \|f\|_{W^k_q(\Omega)} \qquad \text{for all}\ f\in W^k_q(\Omega)\cap L^q_{(0)}(\Omega), k=0,1,
\end{equation}
with a constant $C_q$ depending only on $1<q<\infty$, $\dm$, $k$, and $\Omega$.

Finally, we define $\Div_n\colon L^2(\Omega)\to \Hzero(\Omega)$ by
\begin{equation*}
  \weight{\Div_n f,\varphi}_{\Hzero(\Omega),\Hone(\Omega)}=-(f,\nabla\varphi)_{L^2(\Omega)}\qquad \text{for all}\ \varphi \in \Hone(\Omega).
\end{equation*}
Note that $\Delta_N u=\Div_n \nabla u$ for all $u\in \Hone(\Omega)$.
\medskip

\noindent
{\bf Helmholtz decomposition:} Recall that we have the orthogonal decomposition
\begin{eqnarray*}
  L^2(\Omega)^d&=& L^2_\sigma(\Omega)\oplus G_2(\Omega)\\
G_2(\Omega) &=& \left\{\nabla p\in L^2(\Omega): p\in H^1_{(0)}(\Omega)\right\}.
\end{eqnarray*}
Here $L^2_\sigma(\Omega)$ is the closure of $\{\ue \in C_0^\infty(\Omega)^d:\Div \ue =0\}$ in $L^2(\Omega)^d$. 
The Helmholtz projection $P_\sigma$ is the orthogonal projection onto $L^2_\sigma(\Omega)$.
We note that $P_{\sigma} f = f- \nabla p$, where $p \in H^1_{(0)}(\Omega)$ is the solution of the weak Neumann problem 
\begin{equation}\label{eq:WeakHelmholtz}
  (\nabla p,\nabla \varphi)_{\Omega} = (f, \nabla \varphi)\quad \text{for all}\ \varphi \in C^\infty_{(0)}(\ol{\Omega}).
\end{equation}
We refer to  Simader and Sohr~\cite{SimaderSohr} and Sohr~\cite[Chapter II, Section 2.5]{BuchSohr} for details.

We conclude this section with two technical results related to the Navier boundary condition \eqref{eq:LT5}, which will be needed in Section~\ref{sec:LinearPart}.
\begin{lem}\label{lem:TraceOp}
  Let $\Omega\subset\R^d$ be a bounded domain with $C^k$-boundary, $k\geq 2$. Then there is a first order tangential differential operator $A=\sum_{j=1}^{d}a_j(x)\partial_{\tau_j}$, $a_j\in C^{k-1}(\partial\Omega)$, such that
  \begin{equation}\label{eq:TraceId}
    (\vc{n}\cdot\nabla^2 u)_\tau|_{\partial\Omega}= \nabla_\tau \gamma_1 u + A\gamma_0 u \qquad \text{for all}\ u\in H^2(\Omega),
  \end{equation}
  where $\gamma_j u = \partial_n^j u|_{\partial\Omega}$.
\end{lem}
\begin{proof}
  Since $\nabla_\tau = (I-\vc{n}\otimes \vc{n})\nabla$, we obtain
  \begin{eqnarray*}
    (\vc{n}\cdot\nabla^2 u)_\tau|_{\partial\Omega}&=& \left.(I-\vc{n}\otimes \vc{n})(\vc{n}\cdot \nabla^2 u)\right|_{\partial\Omega}=
(I-\vc{n}\otimes \vc{n})\nabla(\partial_n u) -\sum_{j=1}^d (\partial_{\tau_j}\vc{n})\cdot\nabla u 
  \end{eqnarray*}
  for all $u\in H^2(\Omega)$. Since $\partial_{\tau_j}\vc{n}\cdot \vc{n} = \frac12 \partial_{\tau_j}|\vc{n}|^2=0$, $\partial_{\tau_j}\vc{n}\in C^{k-1}(\partial\Omega)$ is tangential. Therefore \eqref{eq:TraceId} is valid.
\end{proof}

\begin{lem}\label{lem:Extension}
   Let $\Omega\subset\R^d$ be a bounded domain with $C^2$-boundary, $0<T\leq \infty$, $\nu\in C^1(\R)$ with $\inf_{s\in\R}\nu(s)>0$, and $c_0\in H^2(\Omega)$, $d=2,3$.
   \begin{enumerate}
   \item There is a bounded linear operator $E\colon H^{\frac12}(\partial\Omega)^d\to H^2(\Omega)^d$ such that
   \begin{eqnarray*}
     \left.(\vc{n}\cdot 2\tn{D}E \vc{a})_\tau\right|_{\partial\Omega}&=& \vc{a}_\tau, \quad
E \vc{a}|_{\partial\Omega} =0, \quad
     \Div E \vc{a}= 0
   \end{eqnarray*}
   for all $\vc{a}\in H^{\frac12}(\partial\Omega)^d$. Moreover, there is a constant $C>0$ such that $\|Ea\|_{H^1(\Omega)}\leq C\|a\|_{H^{-\frac12}(\partial\Omega)}$ for all $a\in H^{\frac12}(\partial\Omega)^d$.
 \item There is a bounded linear operator 
$$
E_T\colon H^{\frac14,\frac12}(S_T)^d\to L^2(0,T;H^2(\Omega))^d\cap H^1(0,T;L^2_\sigma(\Omega))
$$ 
such that 
   \begin{eqnarray*}
     \left.(\vc{n}\cdot 2\nu(c_0)\tn{D}E_T \vc{a})_\tau\right|_{\partial\Omega}&=& \vc{a}_\tau, \quad
E_T \vc{a}|_{\partial\Omega} =0, \quad
     \Div E_T \vc{a}= 0,\quad E_T \vc{a}|_{t=0}=0
   \end{eqnarray*}
   for all $\vc{a}\in H^{\frac14,\frac12}(S_T)^d$. Moreover, the operator norm of $E_T$ can be estimated independently of $0<T\leq \infty$.
   \end{enumerate}
\end{lem}
\begin{proof}
To prove the first part let $\tilde{\vc{A}}=\tilde{E}\vc{a}\in H^2(\Omega)^d$ such that $\tilde{\vc{A}}|_{\partial\Omega}=0$, $\partial_n \tilde{\vc{A}}|_{\partial\Omega}=\vc{a}_\tau$ and $\|\tilde{\vc{A}}\|_{H^2(\Omega)}\leq C \|\vc{a}\|_{H^{\frac12}(\partial\Omega)}$, $\|\tilde{\vc{A}}\|_{H^1(\Omega)}\leq C \|\vc{a}\|_{H^{-\frac12}(\partial\Omega)}$ for all $\vc{a}\in H^{\frac12}(\partial\Omega)^d$.
If $\Omega=\R^{n-1}\times (0,\infty)$, the existence of such an extension operator $\tilde{E}$ follows e.g. from McLean~\cite[Lemma 3.36]{McLean}. From this the result for a general bounded $C^2$-domain follows by standard localization techniques. 
 
Then we have
  \begin{alignat*}{1}
    (\vc{n}\cdot 2 \tn{D}\tilde{\vc{A}})_\tau|_{\partial\Omega}&= (\nabla_\tau \tilde{\vc{A}}_n  + \partial_n \tilde{\vc{A}}_\tau)|_{\partial\Omega}= 0+\vc{a}_\tau,\\
    \Div \tilde{\vc{A}}|_{\partial\Omega}&= (\Div_\tau \tilde{\vc{A}} + \partial_n \tilde{\vc{A}}_n)|_{\partial\Omega} = 0.    
  \end{alignat*}
Since $\Div \tilde{\vc{A}}|_{\partial\Omega}=0$, $\Div \tilde{\vc{A}}\in H^1_0(\Omega)\cap L^2_{(0)}(\Omega)$ and we can apply the Bogovski-Operator $B$, cf. e.g. \cite{Galdi}, to $\Div \tilde{\vc{A}}$. Hence we obtain $B(\Div \tilde{\vc{A}})\in H^2_0(\Omega)$, $\Div B(\Div\tilde{\vc{A}})=\Div \tilde{\vc{A}} $, and
\begin{alignat*}{1}
  \|B(\Div \tilde{\vc{A}})\|_{H^2(\Omega)}&\leq C\|\tilde{\vc{A}}\|_{H^2(\Omega)}\leq C'\|\vc{a}\|_{H^{\frac12}(\partial\Omega)},\\
  \|B(\Div \tilde{\vc{A}})\|_{H^1(\Omega)}&\leq C\|\tilde{\vc{A}}\|_{H^1(\Omega)}\leq C'\|\vc{a}\|_{H^{-\frac12}(\partial\Omega)}
\end{alignat*}
for all $\vc{a}\in H^{\frac12}(\partial\Omega)$.
Therefore $A:= \tilde{A}- B(\Div \tilde{A})$ has the properties stated above.

Finally, because of \cite[Lemma 2.4]{SolonnikovProceedings}, for every $\vc{a}\in H^{\frac14,\frac12}(S_T)^d$ there is some $\vc{A}\in L^2(0,T;H^1(\Omega))^d\cap H^1(0,T;L^2(\Omega))^d$ such that $\left.(\vc{n}\cdot \nu(c_0)\tn{D}\vc{A})_\tau\right|_{\partial\Omega}
  = \vc{a}_\tau$, $\vc{A}|_{t=0}=0$, $\Div \vc{A}|_{\partial\Omega} =0$, $\vc{A}|_{\partial\Omega}=0$. Moreover, the extension can be chosen such that 
$$
\|\vc{A}\|_{L^2(0,T;H^2)}+\|\vc{A}\|_{H^1(0,T;L^2)}\leq C\|\vc{a}\|_{H^{\frac14,\frac12}(S_T)}
$$ with $C$ independent of $T$ and $\vc{a}$. Analogously to the first part $\Div \vc{A}\in L^2(0,T;H^1_0(\Omega))^d$. Hence $E_T\vc{a}:=\vc{A}-B(\Div \vc{A})\in L^2(0,T;H^2_0(\Omega))^d$. Moreover, due to \cite[Theorem~2.5]{GeissertHeckHieberBogovski} we also have
\begin{equation*}
  \|B(\Div \vc{A})\|_{H^1(0,T;L^2(\Omega))}\leq C\|\Div_n\vc{A}\|_{H^1(0,T;\Hzero(\Omega))}\leq C'\|\vc{A}\|_{H^1(0,T;L^2(\Omega))}, 
\end{equation*}
where $C,C'>0$ are independent of $T$. Altogether $E_T$ has the stated properties.
\end{proof}

\section{Short Time Existence of Strong Solutions}\label{sec:Strong}

In this section we prove existence of strong and unique solutions of system
\eqref{eq:LT1}-\eqref{eq:LT42} and \eqref{eq:LT5}-\eqref{eq:LT7} locally in time in the case $a(c)\equiv m(c)\equiv 1$, i.e., prove Theorem~\ref{thm:main}.   
 As  noted before we will assume that $\beta \neq 0$ since in the case $\beta=0$ the
 linearized system is completely different and short time existence of strong solutions is known in that case, cf. e.g. \cite{ModelH}. In this case we can eliminate the generalized pressure $g_0$ and the chemical potential $\mu$ as follows:

First of all, because of \eqref{eq:rho}, one easily calculates that
\begin{equation}\label{eq:SimpleMixId}
  \frac{\partial \rho}{\partial c} = -\beta \rho^2,\qquad \frac{\partial(\rho c)}{\partial c}= \rho + \frac{\partial
  \rho}{\partial c} c = \alpha \rho^2.
\end{equation}
For the following let $c$ be a sufficiently smooth solution such that $|c(t,x)|\leq 1+\eps_0$, where $\eps_0>0$ is as in Assumption~\ref{assump:3}.
Then (\ref{eq:LT2}) and (\ref{eq:SimpleMixId}) imply
\begin{equation}\label{eq:MaterialDer}
  -\beta \rho^2( \partial_t c + \ve\cdot\nabla c)= - \rho\Div \ve. 
\end{equation}
Combining this with (\ref{eq:LT3}), we obtain the simple identity
\begin{equation}\label{eq:DivvMuId}
  \Div \ve = \beta \Delta \mu_0.
\end{equation}
Thus
 $\mu_0 = \beta^{-1}\Delta_N^{-1} \Div \ve= \beta^{-1}G(\Div \ve)$ since $\vc{n}\cdot \nabla \mu_0|_{\partial\Omega}=0$, where $G(g)$ is 
defined by
\begin{alignat}{2}\label{eq:GDefn1}
  \Delta G(g) &= g &\qquad&\text{in} \ \Omega, \\\label{eq:GDefn2}
  \partial_n G(g) &= 0&&\text{on} \ \partial \Omega,
\end{alignat}
and $\int_\Omega G(g) \sd x=0$. Note that this implies 
\begin{equation}\label{eq:NablaGv}
\nabla G(\Div \ve)= (I-P_\sigma) \ve.   
\end{equation}
Using this and (\ref{eq:LT4}), we can eliminate $g_0$, $\mu_0$ and $\ol{p}(t)$ from the system \eqref{eq:LT1}-\eqref{eq:LT4} and obtain the equivalent system 
 \begin{alignat}{2}\nonumber
   \partial_t \ve +  \ve\cdot \nabla \ve-\rho^{-1}\Div \tn{S}(c,\tn{D}\ve)\\\label{eq:LT1'}
 +\frac1{\beta} \nabla
   \left(\rho^{-2}(\Delta c-\phi(c))\right)&=
 \frac1\beta  G(\Div\ve)\nabla c -\frac1{\beta^2}\nabla
   (\rho^{-1}G(\Div\ve) ) &\ \ &\text{in}\ Q_T,\\\label{eq:LT2'}
    \rho \partial_t c+ \rho \ve\cdot \nabla c &= \beta^{-1} \Div \ve. && \text{in}\ Q_T,
 \end{alignat}
 together with
 \begin{alignat}{2}\label{eq:LT3'}
    \vc{n}\cdot \ve|_{\partial\Omega}= \left.(\vc{n}\cdot \tn{S}(c,\tn{D}\ve))_\tau + \gamma(c) \ve_\tau\right|_{\partial\Omega}= \partial_\normal c|_{\partial\Omega} &=0&\qquad&\text{on}\
   S_T,\\\label{eq:LT4'}
  (\ve,c)|_{t=0}&= (\ve_0,c_0)&\qquad&\text{in}\ \Omega.
 \end{alignat} 
This is indeed an equivalent system since, if $(\ve,c)$ solve the system above, we can simply \emph{define} $g_0$ and $\ol{p}(t)$ by the equation (\ref{eq:LT4}) and $\mu_0$ by $\mu_0 = \beta^{-1}G(\Div \ve)$. Then $\vc{n}\cdot \nabla\mu_0|_{\partial\Omega}$ is automatically satisfied.

We will construct strong solutions by linearizing the
system, proving that the associated linear operator is
an isomorphism between suitable $L^2$-Sobolev spaces, and  applying the
contraction mapping principle to prove existence and uniqueness of the full
system for sufficiently small times. 
 
To this end, let $\tilde{c}_0\in H^1(0,T_0;H^1(\Omega))\cap L^2(0,T_0;H^3(\Omega)\cap H^2_N(\Omega))$ be such that $\tilde{c}_0|_{t=0} = c_0$. The existence of such an $\tilde{c}_0$ follows from Theorem~\ref{thm:LinearPart} below. 
Then \eqref{eq:LT1'}-\eqref{eq:LT4'} are equivalent to
\begin{equation}\label{eq:AbstractEq}
  L(c) 
  \begin{pmatrix}
    \ve\\ \rho (c-\tilde{c}_0)
  \end{pmatrix}
  = \mathcal{F}(\ve,c),
\end{equation}
where for given $c$ the linear operator $L(c)\colon X_T \to Y_T$ is defined by
\begin{eqnarray*}
  L(c) \begin{pmatrix}
    \ve\\ c'
  \end{pmatrix} &=& 
  \begin{pmatrix}
    \partial_t \ve - \Div \widetilde{\tn{S}}(c,\tn{D}\ve) + \frac{1}{\beta\alpha}\nabla \Div (\rho^{-4} \nabla c')\\
     \partial_t c' - \beta^{-1} \Div \ve \\
    \left.(\vc{n}\cdot \widetilde{\tn{S}}(c,\tn{D}(P_\sigma\ve)))_\tau + \tilde{\gamma}(c) (P_\sigma \ve)_\tau\right|_{\partial\Omega}\\
    (\ve,c)|_{t=0}
  \end{pmatrix},\qquad \begin{pmatrix}
    \ve\\ c'
  \end{pmatrix}\in X_T,
  \end{eqnarray*}
  and $c'$ corresponds to $\rho c$.
  Here we have used $\nabla (\rho c)= \alpha \rho^2\nabla c$ and set $\widetilde{\tn{S}}(c,\tn{D}\ve)= 2\tilde{\nu}(c)\tn{D}\ve+\tilde{\eta}(c) \Div \ve\, \tn{I}$,
  $\tilde{\nu}(c)= \hrho(c)^{-1}\nu(c)$, $\tilde{\eta}(c)=
  \hrho(c)^{-1}\eta(c)$, and $\tilde{\gamma}(c)=
  \hrho(c)^{-1}\gamma(c)$. Moreover, 
    $\mathcal{F}\colon X_T \to Y_T$ is a non-linear mapping defined by
  \begin{eqnarray*}
  \mathcal{F}(\ve,c)
  &=& \begin{pmatrix}
    F_1(\ve,c)\\
    -\rho \ve\cdot \nabla  c - \rho\partial_t\tilde{c}_0  + \partial_t \rho (c-\tilde{c}_0)  \\
\left.(\vc{n}\cdot \widetilde{\tn{S}}(c,\nabla^2 G(\Div \ve)))_\tau + \tilde{\gamma}(c) \nabla_\tau G(\Div \ve)\right|_{\partial\Omega}    
\\
    (\ve_0,0)
    \end{pmatrix},\\
    F_1(\ve,c)&=& \frac{ G(\Div\ve) \nabla c}\beta + \frac1 \beta\nabla
    \left(\frac{\pot(c)}{\rho^2}- \frac{G(\Div \ve)}{\beta^2 \rho}-[\rho^{-2},\Div]\nabla c\right)\\
&&+ \frac{1}\beta \nabla (\rho^{-2}\Delta {c}_0) - \ve\cdot \nabla \ve - \nabla \rho^{-1}\cdot \tn{S}(c,\tn{D}\ve),
\end{eqnarray*}
and $ X_T =  X_T^1\times X_T^2$,
\begin{eqnarray*}
X_T^1 &=&\left\{\ue\in H^1(0,T;L^2(\Omega)^d)\cap L^2(0,T;H^2(\Omega)^d): \vc{n}\cdot \ue|_{\partial\Omega}=0 \right\},\\
    X_T^2&=& \left\{ c'\in H^1(0,T;H^1(\Omega))\cap L^2(0,T;H^3(\Omega)): c'|_{t=0}=0, \vc{n}\cdot \nabla c'|_{\partial\Omega}=0\right\},\\ 
  Y_T &=& 
    L^2(Q_T)^d  \times  L^2(0,T;H^1_{0}(\Omega))\times
    \{\vc{a}\in H^{\frac14,\frac12}(S_T): \vc{a}_n =0\}\times
    H^1_n(\Omega)\times H^2_N(\Omega),
\end{eqnarray*}
where $H^{\frac14,\frac12}(S_T):=H^{\frac14}(0,T;L^2(\partial\Omega))\cap L^2(0,T;H^{\frac12}(\partial\Omega))$
Here $[A,B]= AB-BA$ denotes the commutator of two operators. 
The spaces $X_T^1$, $X_T^2$, and $Y_T$ are normed by
\begin{eqnarray*}
  \|\ve\|_{X_T^1} &=& \left\|\left(\partial_t \ve, \nabla^2 \ve\right)\right\|_{L^2(Q_T)} + \|\ve|_{t=0}\|_{H^1(\Omega)},\\
  \|c'\|_{X_T^2} &=& \left\|\left(c', \partial_tc',\partial_t \nabla c', \nabla^3 c'\right)\right\|_{L^2(Q_T)}+\|c'|_{t=0}\|_{H^2(\Omega)},\\
\|(\vc{f},g,\vc{a},\ve_0)\|_{Y_T} &=& \|(\vc{f},\nabla g)\|_{L^2(Q_T)}+\|\vc{a}\|_{H^{\frac14,\frac12}(S_T)}+ \|\ve_0\|_{H^1(\Omega)}+ \|c_0\|_{H^2(\Omega)}.
\end{eqnarray*}

In order to apply the contraction mapping principle to (\ref{eq:AbstractEq})
for sufficiently small $T>0$,
it is essential that $L(c_0)$ is an isomorphism:
\begin{thm}\label{thm:LinearPart}
  Let $c_0\in H^2(\Omega)$, let $T_0>0$, and let Assumption~\ref{assump:3}
  hold true. Then  $L(c_0)\colon X_T \to Y_T$ is an isomorphism for every $0<T\leq T_0$ and there is a constant $C(T_0)>0$  such that
  \begin{equation}\label{eq:InvEstim}
    \|L(c_0)^{-1}\|_{\mathcal{L}(Y_T,X_T)} \leq C(T_0) \qquad \text{for all}\ 0<T\leq T_0.
  \end{equation}
\end{thm}
The proof of this theorem is postponed to Section~\ref{sec:LinearPart}. The second
ingredient for the application of the contraction mapping to
(\ref{eq:AbstractEq}) is the fact that $\mathcal{F}\colon X_T\to Y_T$ is locally
Lipschitz continuous with arbitrarily small Lipschitz constant if $T>0$ is
sufficiently small: 
\begin{prop}\label{prop:LipschitzEstim}
  Let $R>0$ and let Assumption~\ref{assump:3} be satisfied. Then there is a constant $C(T,R)>0$ such that
  \begin{equation*}
    \|\mathcal{F}(\ve_1,c_1)- \mathcal{F}(\ve_2,c_2)\|_{Y_T} \leq C(T,R) \|(\ve_1-\ve_2,c_1-c_2)\|_{X_T} 
  \end{equation*}
  for all $(\ve_j,c_j)\in X_T$ with $\|(\ve_j,c_j)\|_{X_T}\leq R$ and
  $c_j|_{t=0}=c_0$, where $j=1,2$.
  Moreover, $C(T,R)\to 0$ as $T\to 0$.
\end{prop}
\begin{proof}
  Let $F_2(\ve,c)= -\rho \ve\cdot \nabla c$ and let $F_3(\ve,c)=  - \partial_t \rho (c-{c}_0)$.
  For $F'(\ve,c)=(F_1(\ve,c),F_2(\ve,c))$ we will show that
  \begin{equation}\label{eq:F'Estim}
    \|F'(\ve_1,c_1)- F'(\ve_2,c_2)\|_{L^p(0,T;L^2(\Omega)^d\times H^1(\Omega))} \leq C(p,R,T_0)  \|(\ve_1-\ve_2,c_1-c_2)\|_{X_T} 
  \end{equation}
  for all $0<T\leq T_0$ and for some $p>2$. (Note that the third component of $\mathcal{F}$
  is constant.) Then the statement of the proposition for these terms follows from the estimate
  \begin{equation*}
    \|f\|_{L^2(0,T;X)} \leq T^{\frac12-\frac1p} \|f\|_{L^p(0,T;X)},
  \end{equation*}
  where $X$ is an arbitrary Banach space.
  In order to estimate the terms involving $c$, we use that
  \begin{equation}
    \label{eq:LinftyEmbed}
\|c\|_{L^\infty(0,T;H^2(\Omega))}\leq C\|c\|_{X_T^2}
  \end{equation}
due to
  (\ref{eq:BUCEmbedding}) and (\ref{eq:BesselInterpol}), where $C$ is
  independent of $T>0$. 
 Since $H^2(\Omega)$ is an algebra with respect to pointwise multiplication, we have $\widetilde{F}(c)\in L^\infty(0,T;H^2(\Omega)))$ for all $\widetilde{F}\in C^3(\R)$, $c\in X_T^2$, as well as
  \begin{equation}\label{eq:FcEstim}
    \|\widetilde{F}(c_1)-\widetilde{F}(c_2)\|_{L^\infty(0,T;H^2(\Omega))} \leq C(R,\widetilde{F}) \|c_1-c_2\|_{X_T^2}
  \end{equation}
  for all $c_j\in X_T^2$ with $\|c_j\|_{X_T^2}$ $j=1,2$. Hence
  \begin{equation*}
      \|\nabla (\hrho(c_1)^{-2}\phi(c_1)-\hrho(c_2)^{-2}\phi(c_2))\|_{L^\infty(0,T;H^1(\Omega))} \leq C(R,\phi,\hrho) \|c_1-c_2\|_{X_T^2}
  \end{equation*}
  for all $c_j\in X_T^2$ with $\|c_j\|_{X_T^2}\leq R$, $j=1,2$.
  Moreover, 
  \begin{equation*}
    \|G(\Div \ve)\|_{L^\infty(0,T;H^2(\Omega))}\leq C \|\ve\|_{L^\infty(0,T;H^1(\Omega))}\leq C'\|\ve\|_{X^1_T}
  \end{equation*}
 for constants $C,C'$ independent of $T>0$ due to (\ref{eq:W2qEstim}) and (\ref{eq:BUCEmbedding}). 
 Since the product of Lipschitz continuous functions is again Lipschitz continuous, it is sufficient to verify that all the products appearing in $F(v,c)$ are well-defined, which is done as follows:
  \begin{eqnarray*}
    \|G(\Div \ve)\nabla c\|_{L^\infty(0,T;L^2(\Omega))} &\leq& C\|G(\Div \ve)\|_{L^\infty(0,T;H^2(\Omega))}\|c\|_{L^\infty(0,T;H^2(\Omega))}\leq C(R)\\
    \|\nabla (\rho^{-1} G(\Div \ve))\|_{L^\infty(0,T;L^2(\Omega))} &\leq &C\|G(\Div \ve)\|_{L^\infty(0,T;H^2(\Omega))}\|\rho^{-1}\|_{L^\infty(0,T;H^2(\Omega))} \leq C(R)\\
    \|\ve\cdot \nabla \ve\|_{L^4(0,T;L^2(\Omega))} &\leq&
    \|\ve\|_{L^\infty(0,T;L^6(\Omega))}\|\nabla \ve\|_{L^4(0,T;L^3(\Omega))}\\
&\leq& C(R)\|\ve\|_{L^\infty(0,T;H^1(\Omega))}^\frac12\|\nabla \ve\|_{L^2(0,T;L^6(\Omega))}^\frac12  \leq C'(R)\\
    \|\nabla \rho^{-1} \cdot \tn{S}(c,\tn{D}\ve)\|_{L^4(0,T;L^2(\Omega))} &\leq& C\|\nabla \rho^{-1}\|_{L^\infty(0,T;L^6)}\|(\nu(c),\eta(c))\|_{L^\infty(Q_T)} \|\nabla \ve\|_{L^4(0,T;L^3)}\\
&\leq& C(R)\|\nabla \ve\|_{L^\infty(0,T;L^2)}^{\frac12}\|\nabla v\|_{L^2(0,T;L^6)}^\frac12\leq 
C'(R) \\
\|\rho \ve\cdot \nabla c\|_{L^4(0,T;H^1(\Omega))} &\leq& C\|\ve\|_{L^4(0,T;B^{1}_{3,1}(\Omega))} \|\nabla r(c)\|_{L^\infty(0,T;H^1)} \\
&\leq& C\|\ve\|_{L^2(0,T;H^2)}^{\frac12}\|\ve\|_{L^\infty(0,T;H^1)}^{\frac12} \|\nabla r(c)\|_{L^\infty(0,T;H^1)}\\
 &\leq& C(R), 
  \end{eqnarray*}
  for all $(\ve,c)\in X_T$ with $\|(\ve,c)\|_{X_T}\leq R$, where $r'(c)=\hrho(c)$ and  we have used \eqref{eq:BesovProd} and 
  \begin{equation*}
    (H^1(\Omega),H^2(\Omega))_{\frac12,1} = B^{\frac32}_{2,1}(\Omega) \hookrightarrow B^1_{3,1}(\Omega). 
  \end{equation*} 
In order to estimate $\nabla ([\rho^{-2},\Div] \nabla c)$, we use that $\nabla [\rho^{-2},\Div]\nabla c =-\nabla( \nabla \rho^{-2}\cdot \nabla c)= \nabla^2 \rho^{-2}\cdot \nabla c + \nabla\rho^{-2} \cdot \nabla^2 c $ and
  \begin{eqnarray*}
    \|\nabla^2 \rho^{-2} \cdot \nabla c\|_{L^4(0,T;L^2(\Omega))} &\leq& C \|\rho\|_{L^\infty(0,T;H^2(\Omega))}\|\nabla c\|_{L^4(0,T;L^\infty(\Omega))}\\
    &\leq & C(R) \|\nabla c\|_{L^2(0,T;H^2(\Omega))}\|\nabla c\|_{L^\infty(0,T;H^1(\Omega))} \leq C'(R)
  \end{eqnarray*}
  due to $\|f\|_{L^\infty(\Omega)}\leq C\|f\|_{B^{\frac32}_{2,1}(\Omega)}\leq C'\|f\|_{H^1}^\frac12\|f\|_{H^2}^\frac12$. The same estimate holds true for $\nabla\rho^{-2} \cdot \nabla^2 c$.
  Hence (\ref{eq:F'Estim}) holds with $p=4$.

  Furthermore, we estimate $F_3(\ve,c)$ as follows
  \begin{eqnarray*}
    \lefteqn{\|(\hrho(c_1) - \hrho(c_2)) \partial_t \tilde{c}_0\|_{L^2(0,T;H^1(\Omega))}}\\
    &\leq & C\|\hrho(c_1) - \hrho(c_2)\|_{L^\infty(0,T;B^{3/2}_{2,1}(\Omega))} \|\partial_t \tilde{c}_0\|_{L^2(0,T;H^1(\Omega))}
    \leq  C(R,\tilde{c}_0)T^\frac14 \|c_1-c_2\|_{X_T^2}\\
   \lefteqn{ \|(\partial_t \hrho(c_1) - \partial_t\hrho(c_2)) (c_1-\tilde{c}_0)\|_{L^2(0,T;H^1(\Omega))}}\\
    &\leq & C\|\partial_t\hrho(c_1) - \partial_t\hrho(c_2) \|_{L^2(0,T;H^1)}\|c_1-\tilde{c}_0\|_{L^\infty(0,T;B^{3/2}_{2,1})}
    \leq  C(R,\tilde{c}_0)T^\frac14 \|c_1-c_2\|_{X_T^2}\\
  \lefteqn{  \|\partial_t \hrho(c_2) (c_1-c_2)\|_{L^2(0,T;H^1(\Omega))}}\\
    &\leq & C\|\partial_t \rho(c_2)\|_{L^2(0,T;H^1(\Omega))}\|c_1 - c_2\|_{L^\infty(0,T;B^{3/2}_{2,1}(\Omega))}
    \leq  C(R,\tilde{c}_0)T^\frac14 \|c_1-c_2\|_{X_T^2}
  \end{eqnarray*}
  since 
  \begin{equation}
    \label{eq:Embedd}
    X_T^2\hookrightarrow C^\frac12 ([0,T];H^1(\Omega))\cap L^\infty(0,T;H^2(\Omega))\hookrightarrow C^\frac18 ([0,T]; H^\frac74(\Omega)),    
  \end{equation}
 $H^\frac74(\Omega)\hookrightarrow B^{\frac32}_{2,1}(\Omega)$, and $(c_1-c_2)|_{t=0}=(c_1-\tilde{c}_0)|_{t=0}=0$. 

Finally, it remains to estimate the third component of $\mathcal{F}(\ve,c)$. To this end we use that 
\begin{eqnarray*}
  \|(\vc{n}\cdot \nabla^2 G(\Div \ve)_\tau\|_{L^2(0,T;H^{\frac32}(\partial\Omega))}&\leq& C\|\ve\|_{L^2(0,T;H^2(\Omega))}\\
  \|(\vc{n}\cdot \nabla^2 G(\Div \ve)_\tau\|_{H^1(0,T;H^{-\frac12}(\partial\Omega))}&\leq& C\|\ve\|_{H^1(0,T;L^2(\Omega))}
\end{eqnarray*}
for all $\ve\in X_T^1$ since $(\vc{n}\cdot \nabla^2 G(\Div \ve))_\tau= A\gamma_0\ve$ for some first order tangential differential operator $A$, cf. Lemma~\ref{lem:TraceOp}. Hence
\begin{eqnarray*}
  \lefteqn{\|(\vc{n}\cdot \nabla^2 G(\Div \ve))_\tau\|_{L^2(0,T;H^{\frac12}(\partial\Omega))}}\\
&\leq& T^\frac12\|(\vc{n}\cdot \nabla^2 G(\Div \ve)\|_{BUC([0,T];H^{\frac12}(\partial\Omega))}\leq CT^{\frac12}\|\ve\|_{X_T^1}
\end{eqnarray*}
for all $\ve\in X_T^1$ due to \eqref{eq:BUCEmbedding} and
\begin{eqnarray*}
  \|(\vc{n}\cdot \nabla^2 G(\Div \ve)_\tau\|_{H^{\frac14}(0,T;L^2(\partial\Omega))}\leq T^\frac12\|(\vc{n}\cdot \nabla^2 G(\Div \ve)\|_{C^\frac13([0,T];L^2(\partial\Omega))}\leq CT^{\frac12}\|\ve\|_{X_T^1}
\end{eqnarray*}
due to \eqref{eq:VecHEstim} and since $BUC([0,T];H^1)\cap C^{\frac12}([0,T];H^{-\frac12})\hookrightarrow C^{\frac13}([0,T];L^2)$ because of $\|f\|_{L^2}\leq C\|f\|_{H^1}^\frac13\|f\|_{H^{-\frac12}}^{\frac23}$. Together we have
\begin{equation*}
  \|(\vc{n}\cdot \nabla^2 G(\Div \ve)_\tau\|_{H^{\frac14,\frac12}(S_T)}\leq C T^\frac12 \|\ve\|_{X_T^1}
\end{equation*}
By the same arguments one shows that 
\begin{equation*}
  \|( \nabla G(\Div \ve))_\tau\|_{H^{\frac14,\frac12}(S_T)}\leq C T^\frac12 \|\ve\|_{X_T^1}. 
\end{equation*}
Moreover, using \eqref{eq:Embedd},
 $\|fg\|_{H^{\frac12}(\partial\Omega)}\leq C\|f\|_{B^1_{2,1}(\partial\Omega)}\|g\|_{H^{\frac12}(\partial\Omega)}$ as well as $\|fg\|_{L^2(\partial\Omega)}\leq \|f\|_{L^4(\partial\Omega)}\|g\|_{L^4(\partial\Omega)}\leq C\|f\|_{H^{\frac12}(\partial\Omega)}\|g\|_{H^{\frac12}(\partial\Omega)}$, one obtains
 \begin{eqnarray*}
 \lefteqn{\| (\tilde{\nu}(c_1)-\tilde{\nu}(c_2))a\|_{H^{\frac14}(0,T;L^2(\partial\Omega))}}\\
 &\leq& CT^{\frac14}\|(\tilde{\nu}(c_1)-\tilde{\nu}(c_2))\|_{C^{\frac12}([0,T];H^{\frac12}(\partial\Omega))}\|a\|_{L^2(0,T;H^{\frac12}(\partial\Omega))}\\
&& + CT^{\frac14}\|(\tilde{\nu}(c_1)-\tilde{\nu}(c_2))\|_{C^{\frac14}([0,T];B^{1}_{2,1}(\partial\Omega))}\|a\|_{H^\frac14(0,T;L^2(\partial\Omega))}
 \end{eqnarray*}
and therefore
\begin{eqnarray}\nonumber
 \lefteqn{\| (\tilde{\nu}(c_1)-\tilde{\nu}(c_2))a\|_{H^{\frac14,\frac12}(S_T)}}\\\label{eq:TraceEstim}
&\leq& C(R,\tilde{\nu})T^{\frac14}\|c_1-c_2\|_{BUC([0,T];H^2(\Omega))\cap C^{\frac12}([0,T];H^1(\Omega))}\|a\|_{H^{\frac14,\frac12}(S_T)}
\end{eqnarray}
for all $a\in H^{\frac14,\frac12}(S_T)$ and $c_j\in X_T^2$ with $\|c_j\|_{X_T^2}\leq R$ and $c_j|_{t=0}=c_0$, $j=1,2$. Combining these estimates, we obtain
\begin{eqnarray*}
  \lefteqn{\|(\tilde{\nu}(c_1)\vc{n}\cdot \nabla^2 G(\Div \ve_1))_\tau- (\tilde{\nu}(c_2)\vc{n}\cdot \nabla^2 G(\Div \ve_2))_\tau\|_{H^{\frac14,\frac12}(S_T)}}\\
&\leq & \|((\tilde{\nu}(c_1)-\tilde{\nu}(c_2))\vc{n}\cdot \nabla^2 G(\Div \ve_1))_\tau\|_{H^{\frac14,\frac12}(S_T)}\\
&&+\|(\tilde{\nu}(c_2)\vc{n}\cdot \nabla^2 G(\Div (\ve_1-\ve_2))_\tau\|_{H^{\frac14,\frac12}(S_T)}\\
&\leq & C(R)T^{\frac14}\left(\|\ve_1-\ve_2\|_{X_T^1}+\|c_1-c_2\|_{X_T^2}\right)
\end{eqnarray*}
for all $u_j=(\ve_j,c_j)\in X_T$ with $\|u_j\|_{X_T}\leq R$ and $c_j|_{t=0}=c_0$, $j=1,2$  since $\|(\nabla \ve,\ve)\|_{H^{\frac14,\frac12}(S_T)}\leq C\|\ve\|_{X^1_T}$. In the same way one can estimate $\tilde{\gamma}(c_1)(\nabla G(\Div\ve_1)_\tau- \tilde{\gamma}(c_2)(\nabla G(\Div\ve_2)_\tau|_{\partial\Omega}$, which proves the necessary estimate of the third component of $\mathcal{F}$.

Altogether this proves the proposition.
\end{proof}

Combining Theorem~\ref{thm:LinearPart} and
Proposition~\ref{prop:LipschitzEstim}, we are now able to prove our main result.\\
\begin{proof*}{of Theorem~\ref{thm:main}}
  First of all, let $\tilde{c}_0\in X^2_1$ be such that $\|\tilde{c}_0\|_{X^2_1}\leq
  C'\|c_0\|_{H^2(\Omega)}$ for some fixed constant $C'>0$ and let
 $$
R=\max\left(C'\|c_0\|_{H^2(\Omega)},\|S(\tilde{c}_0) L^{-1}(\tilde{c}_0)\mathcal{F}(0)\|_{X_T}\right),
$$
  where $S$ is defined below.  
  We will construct a solution in $B_R(0)\subset X_T$ for sufficiently small
  $0<T\leq 1$ with the aid of the contraction mapping principle. If
  \begin{equation*}
  c_1,c_2\in B_R(0)\subset BUC([0,T];H^2(\Omega))\cap C^\frac12(0,T;H^1(\Omega))=:\widetilde{X}_T^2,
  \end{equation*}
 $0<T\leq 1$ with $c_1|_{t=0}= c_2|_{t=0}$, we conclude
  \begin{eqnarray*}
    \lefteqn{\left\|L(c_1)u-L(c_2)u\right\|_{Y_T}}\\
    &\leq&
    C\left(\|\widetilde{\tn{S}}(c_1,\tn{D}\ve)-\widetilde{\tn{S}}(c_2,\tn{D}\ve)\|_{L^2(0,T;H^1)}+ \left\|\left.\widetilde{\tn{S}}(c_1,\tn{D}\ve)-\widetilde{\tn{S}}(c_2,\tn{D}\ve)\right|_{\partial\Omega}\right\|_{H^{\frac14,\frac12}(S_T)}\right.\\
&&\left.+ \left\|\left.(\tilde{\gamma}(c_1)-\tilde{\gamma}(c_2))\ve\right|_{\partial\Omega}\right\|_{H^{\frac14,\frac12}(S_T)}+\|(\hrho(c_1)^{-2}-\hrho(c_2)^{-2})\nabla c'\|_{L^2(0,T;H^2)}\right)\\
&\leq & C\left(\|c_1-c_2\|_{L^\infty(0,T;B^{1}_{3,1})}\|\ve\|_{X_T^1}
+
    \|\hrho(c_1)^{-2}-\hrho(c_2)^{-2}\|_{L^\infty(0,T;H^2)}\|\nabla c'\|_{L^2(0,T;B^1_{3,1})}\right.\\
&&\left. +\|\hrho(c_1)^{-2}-\hrho(c_2)^{-2}\|_{L^\infty(0,T;B^1_{3,1})}\|c'\|_{L^2(0,T;H^3)}\right)\\
&\leq &C(R)\left(T^\frac14\|c_1-c_2\|_{\widetilde{X}_T^2} \|u\|_{X_T}+T^\frac14\|c_1-c_2\|_{BUC([0,T];H^2)}\|c'\|_{L^4([0,T];B^{\frac52}_{2,1})}\right)\\
    &\leq& C(R) T^\frac14\left(\|c_1-c_2\|_{C^\frac12([0,T];H^1)}+\|c_1-c_2\|_{BUC([0,T];H^2)}\right)\|u\|_{X_T}.
  \end{eqnarray*} 
  where $u=(\ve,c')\in X_T$ with $c'|_{t=0}=0$. Here we have {}
  used that 
$$
\|fg\|_{H^2(\Omega)}\leq C\left(\|f\|_{B^1_{3,1}(\Omega)}\|g\|_{H^2(\Omega)}+\|f\|_{H^2(\Omega)}\|g\|_{B^1_{3,1}(\Omega)}\right)
$$ 
and have used \eqref{eq:TraceEstim}. Hence there is some $0<T_0\leq 1$ such that
  \begin{eqnarray*}
    \left\|L(c)u-L(c_0)u\right\|_{Y_T} &\leq& \frac1{4C(1)}\|u\|_{X_T}\quad \text{for
    all}\ 0<T\leq T_0, \|c\|_{X_T^2}\leq R \\
  \left\|L(c_1)u-L(c_2)u\right\|_{Y_T} &\leq& \frac1{4C(1)}\|u\|_{X_T}\quad
    \text{for all}\ 0<T\leq T_0, \|c_j\|_{X_T^2}\leq R, j=1,2, 
  \end{eqnarray*}
  since $c_0\in \widetilde{X}_T^2$. This implies that $L(c)\colon Y_T \to X_T$ is invertible and 
  $\|L(c)^{-1}\|_{\mathcal{L}(Y_T,X_T)}\leq \frac43 C(1)\leq 2 C(1)$ as well as
  \begin{eqnarray*}
    \lefteqn{\|L(c_1)^{-1}-L(c_2)^{-1}\|_{\mathcal{L}(Y_T,X_T)}}\\
&\leq&
    4C(1)^2\|L(c_1)-L(c_2)\|_{\mathcal{L}(X_T,Y_T)}
    \leq C(R)T^\frac18\|c_1-c_2\|_{X_T^2}.
  \end{eqnarray*}
  Moreover, we can choose $T_0$ so small that $\|c\|_{X_{T_0}^2}\leq R$ and $c|_{t=0}=c_0$ implies $\|c-c_0\|_{C^0(\ol{Q_T})}\leq \eps_0$ since $X_{T_0}^2\hookrightarrow C^{\frac18}([0,T_0];C^0({\ol{\Omega}}))$, where $\eps_0$ is as in Assumption~\ref{assump:3}.
  Then $|c(x,t)|\leq 1+\eps_0$ for all $0\leq t\leq T, x\in\ol{\Omega}$.

  Hence we can write (\ref{eq:LT1'})-(\ref{eq:LT4'}) as a fixed point equation
  \begin{equation*}
    u = S(c)L^{-1}(c) \mathcal{F}(u)=:\mathcal{G}(u), 
  \end{equation*}
  where $S(c)\colon X_T\to X_T$ is defined by 
  \begin{equation*}
S(c)
  \begin{pmatrix}
    \ve'\\ c'
  \end{pmatrix}
= 
\begin{pmatrix}
\ve'\\ \hrho(c)^{-1}c'+\tilde{c}_0
\end{pmatrix}
\equiv S'(c)+
\begin{pmatrix}
  0\\ \tilde{c}_0
\end{pmatrix}.
  \end{equation*}
 In order to estimate $S(c)$, we use 
\begin{eqnarray*}
  \|(\hrho(c_1)^{-1}-\hrho(c_2)^{-1})c'\|_{L^2(0,T;H^3(\Omega))}&\leq& C(R)T^{\frac14} \|c_1-c_2\|_{X^2_T}\|c'\|_{X^2_T}
\end{eqnarray*}
 if $\|c_j\|_{X^2_T}\leq R$, which can be shown using that $X^2_T\hookrightarrow C^{\frac14}([0,T];B^{3/2}_{2,1}(\Omega))$ and $c_1-c_2|_{t=0}=c'|_{t=0}=0$.
 Moreover, we have
\begin{eqnarray*}
  \lefteqn{\|\partial_t((\hrho(c_1)^{-1}-\hrho(c_2)^{-1})c')\|_{L^2(Q_T)}}\\
&\leq& 
  \|(\hrho(c_1)^{-1}-\hrho(c_2)^{-1})\partial_t c'\|_{L^2(Q_T)}+
  \|(\partial_t (\hrho(c_1)^{-1}-\hrho(c_2)^{-1})c'\|_{L^2(Q_T)}\\
  &\leq& C(R)\left(\|c_1-c_2\|_{BUC([0,T];C(\ol{\Omega}))}\|c'\|_{X_T^2}
+ \|\partial_t(c_1-c_2)\|_{L^2(Q_T)}\|c'\|_{BUC([0,T];C(\ol{\Omega}))}\right)\\
&\leq& C(R)T^{\frac14} \|c_1-c_2\|_{X_T^2}\|c'\|_{X_T^2},
\end{eqnarray*}
provided that $\|c_j\|_{X_T^2}\leq R$ and $c_1-c_2|_{t=0}=c'|_{t=0}=0$.
Here we have used that
\begin{equation*} 
  \|d\|_{BUC([0,T];C(\ol{\Omega}))}\leq CT^{\frac14}\|d\|_{C^{\frac14}([0,T];B^{3/2}_{2,1}(\Omega))}\leq C'T^{\frac14}\|d\|_{X_T^2} 
\end{equation*}
for all $d\in X_T^2$ with $d|_{t=0}=0$.
Altogether this implies
\begin{equation*}
  \|S(c_1)-S(c_2)\|_{\mathcal{L}(X_T^0,X_T)}\leq C(R) T^{\frac14}\|c_1-c_2\|_{X_T^2}
\end{equation*}
provided that $\|c_j\|_{X_T^2}\leq R$ and $c_1-c_2|_{t=0}=0$, where $X_T^0=\{(\vc,c')\in X_T: c'|_{t=0}=0\}$. 
Therefore we get
  \begin{eqnarray*}
    \lefteqn{\|S(c_1)L^{-1}(c_1)
      -S(c_2)L^{-1}(c_2)\|_{\mathcal{L}(Y_T^0,X_T)}}\\
&\leq&
    \|S'(c_1)\|_{\mathcal{L}(X_T)}\|L^{-1}(c_1)
    -L^{-1}(c_2)\|_{\mathcal{L}(Y_T,X_T)}\\
&&+
    \|S(c_1)-S(c_2)\|_{\mathcal{L}(X_T^0,X_T)}\|L^{-1}(c_2)\|_{\mathcal{L}(Y_T,X_T)}\\
    &\leq& C(R) T^\frac14 \|c_1-c_2\|_{X_T^2}, 
  \end{eqnarray*}
where  $Y_T^0=\{(\vc{f},g,a,\ve_0,c'_0)\in Y_T: c'_0=0\}$. -- Note that $L(c)^{-1}(Y_T^0)=X_T^0$. --
  Because of Proposition~\ref{prop:LipschitzEstim}, we have
  \begin{eqnarray*}
    \lefteqn{\|S(c_1)L^{-1}(c_1) \mathcal{F}(u_1)-S(c_2)L^{-1}(c_2)\mathcal{F}(u_2)\|_{X_T}}\\
 &\leq& \|S(c_1)L^{-1}(c_1)-S(c_2)L^{-1}(c_2)\|_{\mathcal{L}(Y_T^0,X_T)}\|\mathcal{F}(u_1)\|_{Y_T} + C \|\mathcal{F}(u_1)-\mathcal{F}(u_2)\|_{X_T}\\
     &\leq& C(R) T^\frac14\|u_1-u_2\|_{X_T} +
    C(T,R)\|u_1-u_2\|_{X_T} \leq \frac12 \|u_1-u_2\|_{X_T}
  \end{eqnarray*} 
  for all sufficiently small $0<T\leq T_0$
  and all $u_j=(\ve_j,c_j)\in X_T$ with $\|u_j\|_{X_T}\leq R$ and $c_j|_{t=0}=c_0$. 
Moreover,
\begin{eqnarray*}
  \lefteqn{\|S(c) L^{-1}(c)\mathcal{F}(u)\|_{X_T}}\\
 &\leq& \|S(c) L^{-1}(c)\mathcal{F}(u)-S(\tilde{c}_0) L^{-1}(\tilde{c}_0)\mathcal{F}(0)\|_{X_T}+ \|S(\tilde{c}_0) L^{-1}(\tilde{c}_0)\mathcal{F}(0)\|_{X_T} \\
&\leq &\frac12 \|u\|_{X_T}+ \frac{R}2\leq R
\end{eqnarray*}
for all $\|u\|_{X_T}\leq R$. Hence by the contraction mapping principle there is a unique solution $u=(\ve,c)$ of \eqref{eq:LT1'}-\eqref{eq:LT4'} with $\|u\|_{X_T}\leq R$.

Thus we have proved that for every $\ve_0\in H^1_n(\Omega),c_0\in H^2(\Omega)$ there are some $T,R>0$ such that the system (\ref{eq:LT1'})-(\ref{eq:LT4'}) has a unique solution $(\ve,c)\in X_T$ with $\|(\ve,c)\|_{X_T}\leq R$. In order to show that there is only one solution $(\ve,c)\in X_T$ of (\ref{eq:LT1'})-(\ref{eq:LT4'}), let $(\ve',c')\in X_T$ be a second solution of (\ref{eq:LT1'})-(\ref{eq:LT4'}) and let $R'=\max(R,\|(\ve',c')\|_{X_T})$. Then by the arguments above there is some $T'\in (0,T]$ such that (\ref{eq:LT1'})-(\ref{eq:LT4'}) has a unique solution $(\ve'',c'')\in X_{T'}$ (on the time interval $(0,T')$) with $\|(\ve'',c'')\|_{X_{T'}}\leq R'$. Hence $(\ve,c)|_{(0,T')}\equiv (\ve',c')|_{(0,T')}\equiv (\ve'',c'')$. Repeating this argument finitely many times (with a shift in time), we conclude that $(\ve,c)\equiv (\ve',c')$ on the full time interval $(0,T)$. 
\end{proof*}

\section{Linearized System -- Proof of Theorem~\ref{thm:LinearPart}}\label{sec:LinearPart}

In this section we will show unique solvability of the linear system
\begin{alignat}{2}\label{eq:lin1'}
      \partial_t \ve - \Div \widetilde{\tn{S}}(c_0,\tn{D}\ve) + \frac{\eps}{\beta\alpha}\nabla \Div (\rho^{-4} \nabla c')&= \vc{f}_1&\quad& \text{in}\ Q_T,\\\label{eq:lin2'}
     \partial_t c' - \beta^{-1} \Div \ve&= f_2 &\quad& \text{in}\ Q_T,\\\label{eq:lin3'}
  \left.(\vc{n}\cdot \widetilde{\tn{S}}(c_0,\tn{D}(P_\sigma\ve)))_\tau+\tilde{\gamma}(c_0) (P_\sigma\ve)_\tau\right|_{\partial\Omega}
  &= \vc{a} &\quad& \text{on}\ S_T,\\\label{eq:lin3b'}
  \vc{n}\cdot\ve|_{\partial\Omega}=\partial_n c|_{\partial\Omega} &= 0 &\quad& \text{on}\ S_T,\\\label{eq:lin4'}
  (\ve,c')|_{t=0} &=(\ve_0,c_0') && \text{in}\ \Omega,
\end{alignat}
where $(\vc{f}_1,f_2,\vc{a},\ve_0,c_0')\in Y_T$, $(\ve,c')\in X_T$, and $X_T,Y_T$ are as in Section~\ref{sec:Strong}. 

First of all, we can reduce to the case $\vc{a}\equiv  0$ by subtracting from $\ve$ some $\vc{w}\in X^1_T$ such that $\left.(\vc{n}\cdot \widetilde{\tn{S}}(c_0,\tn{D}\we))_\tau+\tilde{\gamma}(c_0) \we_\tau\right|_{\partial\Omega}
  = \vc{a}_\tau$, $\vc{w}|_{t=0}=0$, $\Div \we =0$, $\we|_{\partial\Omega}=0$. The existence of such a $\vc{w}\in X^1_T$, depending continuously on $\vc{a}\in H^{\frac14,\frac12}(S_T)^d$ with $\vc{a}_n=0$, follows directly from Lemma~\ref{lem:Extension}. 
For the following let 
\begin{equation*}
  T_\gamma \ue = \left.(\vc{n}\cdot \widetilde{\tn{S}}(c_0,\tn{D}\ue))_\tau+\tilde{\gamma}(c_0) \ue_\tau\right|_{\partial\Omega}\quad \text{for all}\ \ue \in H^2(\Omega)^d.
\end{equation*}

Now we will reformulate the system above in an appropriate way assuming that $(\ve,c')\in X_T$. Since \eqref{eq:lin2'} depends only on $\Div \ve$, we will use the Helmholtz decomposition to decompose $\ve$. More precisely, using  $L^2(\Omega)^d= L^2_\sigma(\Omega) \oplus G_2(\Omega)$ and  applying 
  $P_\sigma$ and $(I-P_\sigma)$ to \eqref{eq:lin1'}, 
  we obtain that \eqref{eq:lin1'} is equivalent to
  \begin{eqnarray}\label{eq:Split1}
    \partial_t \we - P_\sigma\Div \widetilde{\tn{S}}(c_0,\tn{D}\we) - P_\sigma\Div \widetilde{\tn{S}}(c_0,\nabla^2 G(\Div \ve)) &=& P_\sigma \vc{f}_1,  \\\label{eq:Split2}
    \partial_t \nabla G(\Div \ve) - (I-P_\sigma)\Div \widetilde{\tn{S}}(c_0,\tn{D}\ve) +
    \frac\eps{\alpha\beta} \nabla \Div (\rho^{-4}_0 \nabla c') &=& (I-P_\sigma)\vc{f}_1, 
  \end{eqnarray}
  where $\we=P_\sigma \ve$ and $\ve=\we+\nabla G(\Div\ve)$ and $G$ is defined by \eqref{eq:GDefn1}-\eqref{eq:GDefn2}. In the following let $g=\Div
  \ve$.

In order to determine the principal part of \eqref{eq:Split1}, we use that
  \begin{eqnarray}\nonumber
    \lefteqn{P_\sigma\Div \widetilde{\tn{S}}(c_0,\nabla^2 G(\Div\ve)) = 
    P_\sigma\Div \left(2\tilde{\nu}(c_0)\nabla^2 G(\Div\ve)\right) + P_\sigma \nabla
    (\tilde{\eta}(c_0) \Div \ve)}\\\nonumber
  &=& P_\sigma\nabla \Div \left((2\tilde{\nu}(c_0)) \nabla G(\Div\ve) \right)
  -P_\sigma \Div (2\nabla\tilde{\nu}(c_0)\otimes \nabla G(\Div\ve))\\\label{eq:DefnB1}
  &=& -P_\sigma \Div (2\nabla\tilde{\nu}(c_0)\otimes \nabla G(\Div\ve)) \equiv B_1 g.
\end{eqnarray}
Moreover,
testing (\ref{eq:Split2}) with $\nabla \varphi$, where $\varphi \in C_0^\infty(0,T;H^1_{(0)}(\Omega))$, one sees that (\ref{eq:Split2}) is equivalent to
  \begin{eqnarray*}
    \lefteqn{-\weight{\partial_t g, \varphi}_{L^2(0,T;\Hzero)} - \left(\Div \widetilde{\tn{S}}(c_0,\tn{D}\ve),\nabla \varphi\right)_{Q_T}}\\
&& + \frac{\eps}{\alpha\beta} (\nabla \Div(\rho^{-4}_0 \nabla \tilde{c}), \nabla \varphi)_{Q_T} = ((I-P_\sigma) \vc{f}_1,\nabla \varphi)_{Q_T} 
  \end{eqnarray*}
  for all $\varphi \in C_0^\infty(0,T;H^1_{(0)}(\Omega))$, where we have used again the orthogonal decomposition $L^2(\Omega)^d= L^2_\sigma(\Omega) \oplus G_2(\Omega)$ and the fact that $g=\Div_n \ve\in H^1(0,T;\Hzero(\Omega))$ if $\ve\in X_T^1$. Moreover, we note that
  \begin{eqnarray*}
    \lefteqn{\Div \widetilde{\tn{S}}(c_0,\nabla^2 G(\Div \ve)) = \Div (2\tilde{\nu}(c_0) \nabla^2 G(\Div \ve)) + \nabla (\tilde{\eta}(c_0) g)}\\
 &=&2\tilde{\nu}(c_0) \nabla g + \nabla (\tilde{\eta}(c_0) g) + 2(\nabla \tilde{\nu}(c_0)) \cdot \nabla^2 G(\Div \ve)\\  
 &=&\nabla((2\tilde{\nu}(c_0)+\tilde{\eta}(c_0)) g) + 2(\nabla \tilde{\nu}(c_0)) \cdot \nabla^2 G(\Div \ve)- (2\nabla \tilde{\nu}(c_0)) g.
  \end{eqnarray*}
  Hence
  \begin{equation*}
    \left(\Div \widetilde{\tn{S}}(c_0,\tn{D} \nabla G(g)),\nabla \varphi\right) = -\weight{\Delta_N (a(c_0)g), \varphi}_{\Hzero,\Hone} + \weight{B_2 g, \varphi}_{\Hzero,\Hone},
  \end{equation*}
  where $a(c_0)= 2\tilde{\nu}(c_0)+\tilde{\eta}(c_0)$ and $B_2$ is defined by the equation.

Therefore we can reformulate \eqref{eq:lin2'}-\eqref{eq:Split2} with $\vc{a}\equiv 0$ more abstractly as 
\begin{eqnarray}\label{eq:AbstractEvo}
  \partial_t u + \A u + \mathcal{B}u 
 &=& 
  \begin{pmatrix}
f_2\\
    \Div_n (I-P_\sigma) \vc{f}_1 \\
     P_\sigma \vc{f}_1
  \end{pmatrix}=: f\\\label{eq:AbstractEvo2}
u|_{t=0}&=&
\begin{pmatrix}
  c'_0\\ g_0\\\vc{w}_0
\end{pmatrix}=: u_0
\end{eqnarray}
where $u= (c',g,\we)^T$, $g_0=\Div \ve_0$,  $\ve_0=\we_0+ \nabla G(g_0) $, and
\begin{eqnarray*}
  \mathcal{A} u &=& \left(
  \begin{array}{cc}
    A_1 
    & A_2
 \\ 0&-P_\sigma \Div \widetilde{\tn{S}}(c_0,\tn{D}\we)
  \end{array}
  \right) 
  \begin{pmatrix}
    (c',g)^T\\ \we
  \end{pmatrix},\\
  A_1 
&=& \left(
      \begin{array}{cc}
  0& -\beta^{-1}P_0  \\
  \frac\eps{\alpha\beta} \Delta_N(\Div \rho^{-4}_0\nabla\cdot)  &       -\Delta_N (a(c_0)  \cdot)
      \end{array}
\right) 
,\\
  A_2 \we&=& 
  \begin{pmatrix}
   0\\  -\Div_n (\Div \widetilde{\tn{S}}(c_0, \tn{D}\we))
  \end{pmatrix},\
  \mathcal{B} u = 
  \begin{pmatrix}
  0\\ B_2 g\\ 
-B_1 g
  \end{pmatrix}.
\end{eqnarray*}
and the domains of $\A,\B$ are defined as
\begin{eqnarray*}
  \mathcal{D}(\A) &=& \mathcal{D}(\B)= \left\{ (c',g,\we)^T: (c',g)\in \mathcal{D}(A_1), \we\in H^2(\Omega)^d\cap L^2_\sigma(\Omega):T_\gamma\we =0\right\},\\
  \mathcal{D}(A_1) &=& \left(H^3(\Omega)\cap H^2_N(\Omega)\right)\times H^1_{(0)}(\Omega).
\end{eqnarray*}

We consider $\A$ and $\B$ as unbounded operators on 
\begin{equation*}
  H= H_1 \times L^2_\sigma(\Omega) \quad \text{where}\ H_1=  H^1(\Omega)\times \Hzero(\Omega).
\end{equation*}

For the following analysis it will be crucial that $B_1$ and $B_2$ are of lower order compared to $\A$. More precisely, we have:
\begin{lem}\label{lem:LowerOrder}
  Let $s\in (\frac12,1]$ and $c_0\in H^2(\Omega)$. Then there are constants $C(c_0)$, $C'(c_0,s)>0$ such that
  \begin{alignat}{2}
  \|B_1g\|_{L^2(\Omega)}\label{eq:B1Estim}
&\leq C'(s,c_0) \|g\|_{H^s(\Omega)}, &\qquad&\\
 \|B_2g\|_{H^{-1}_{(0)}(\Omega)}&\leq \label{eq:B2Estim}
  C(c_0) \|g\|_{H^\frac12(\Omega)} 
  \end{alignat}
for all $g\in \Hone(\Omega)$,
where $B_1,B_2$ are as above.
\end{lem}
\begin{proof}
By the definition of $B_1$ and \eqref{eq:DefnB1}, we have
  \begin{eqnarray*}
  \|B_1g\|_{L^2(\Omega)}&=&  \|P_\sigma\Div \widetilde{\tn{S}}(c_0,\tn{D}\nabla
 G(g))\|_{L^2(\Omega)}\\
&\leq& C\|\nabla \tilde{\nu}\|_{H^1(\Omega)}\|\nabla G(g)\|_{H^{1+s}(\Omega)} \leq C(s,c_0) \|g\|_{H^s(\Omega)} 
  \end{eqnarray*}
  for every $s>\frac12$ since $c_0\in H^2(\Omega)$, $\Delta_N^{-1}\colon
  H^s(\Omega)\to H^{s+2}(\Omega)$ for all $s\in [0,1]$ due to (\ref{eq:W2qEstim}),
  and $\|fg\|_{H^1}\leq C_s\|f\|_{H^s}\|g\|_{H^1}$ if $s>\frac{d}2$.

Finally, $B_2$ satisfies
  \begin{eqnarray*}
    \|B_2g\|_{H^{-1}_{(0)}(\Omega)}&\leq& 2\|(\nabla \tilde{\nu}(c_0)) \cdot \nabla^2 G(g)-\nabla \tilde{\nu}(c_0) g\|_{L^2(\Omega)}\\
 &\leq& C'\|\nabla c_0\|_{L^6(\Omega)} \|g\|_{L^3(\Omega)} \leq C(c_0) \|g\|_{H^\frac12(\Omega)}
  \end{eqnarray*}
  due to (\ref{eq:W2qEstim}). 
\end{proof}

Because of the triangle structure of $\A$, it is sufficient to prove that $-A_1$ and $P_\sigma(\Div \widetilde{\tn{S}}(c_0,\cdot))$ generate analytic semigroups in order to have the same for $-\A$.
\begin{lem}\label{lem:A1}
  Let $A_1$ and $H_1$ be as above. Then $-A_1$ generates a bounded analytic semigroup.
 Moreover, $\|A_1 (c',g)^T\|_{H_1}+ |m(c')|$ is equivalent to $\|g\|_{H^1}+ \|c'\|_{H^3}$.
\end{lem}
\begin{proof}
Let $H_0= H^1_{(0)}(\Omega)\times \Hzero(\Omega)$. Then $A_1$ leaves $H_0$ invariant. We first show that $-A_1|_{H_0}$ generates an analytic semigroup. To this end we use that
\begin{eqnarray*}
  -A_1 |_{H_0} &=& \left(
    \begin{array}{cc}
      0& \beta^{-1}I  \\
       -A &    -B
    \end{array}
  \right),\quad \text{where}\ 
  A= \frac\eps{\alpha\beta} \Delta_N\Div (\rho^{-4}_0\nabla \cdot),
  B= -\Delta_N (a(c_0)\cdot). 
\end{eqnarray*}
Here $\mathcal{D}(A)= H^3(\Omega)\cap H^2_N(\Omega)\cap H^1_{(0)}(\Omega)$, $\mathcal{D}(B)= H^1_{(0)}(\Omega)$.
Without loss of generality let $\beta=1$. (Otherwise replace $A_1$ by $\beta A_1$.) Because of \cite[Theorem 1.1]{ChenTriggiani}, $-A_1|_{H_0}$ generates an analytic semigroup on $H_0$ provided that the following conditions are satisfied:
\begin{enumerate}
\item[{\bf H1}] $A,B$ are positive self-adjoint operators on $\Hzero(\Omega)$ with dense domains $\mathcal{D}(A)$, $\mathcal{D}(B)$. $A$ has a compact resolvent.
\item[{\bf H2}] $\mathcal{D}(A^\frac12)=\mathcal{D}(B)$ and there are constants $0<\rho_1<\rho_2 <\infty$ such that $\rho_1 A^{\frac12}\leq B \leq \rho_2 A^{\frac12}$.
\end{enumerate}
We note that, if these conditions are satisfied, then $e^{-A_1|_{H_0}t}$ is an exponentially decreasing semigroup of  contractions on $H_0$ equipped with the norm of $\mathcal{D}(A^{\frac12})\times \Hzero(\Omega)$. In particular, $A_1$ is invertible. 

Let us verify the conditions above. First of all, $A,B$ are positive and symmetric since
\begin{eqnarray*}
  (Au,v)_{\Hzero} &=& -\frac{\alpha\beta}\eps\left(\Div(\rho_0^{-4}\nabla u),v\right)_{L^2(\Omega)} = \frac{\alpha\beta}\eps\left( \rho_0^{-4}\nabla u,\nabla v\right)_{L^2(\Omega)} = (u,Av)_{\Hzero}\\
(Bu',v')_{\Hzero} &=& (a(c_0)u',v')_{L^2(\Omega)} =(u',Bv')_{\Hzero}
\end{eqnarray*}
for all $u,v\in \mathcal{D}(A)$, $u',v'\in\mathcal{D}(B)$,
where we have used \eqref{eq:DeltaNId}. Moreover, with the aid of the Lemma of Lax-Milgram and standard elliptic regularity theory one easily shows that $A$ and $B$ are invertible. Hence $A,B$ are self-adjoint. 
In order to verify {\bf H2}, we use that there are constants $c_0,C_0>0$ such that
\begin{equation*}
  c_0 (\nabla u,\nabla u)_{L^2(\Omega) }\leq (A u,u)_{\Hzero}= \frac{\alpha\beta}\eps\left( \rho_0^{-4}\nabla u,\nabla u\right)_{L^2(\Omega)}\leq C_0 (\nabla u,\nabla u)_{L^2(\Omega) }
\end{equation*}
since $\rho_0^{-4}$ is bounded above and below,
where 
\begin{equation*}
  (\nabla u,\nabla u)_{L^2}= -(\Delta_N u,u)_{L^2}= ((-\Delta_N)^2 u,u)_{\Hzero}=\|\Delta_N u\|_{\Hzero}^2.
\end{equation*}
Hence $c_0(-\Delta_N)^2\leq A \leq C_0 (-\Delta_N)^2$ in $\Hzero(\Omega)$. This implies that there are $c_1,c_2>0$ 
\begin{eqnarray*}
  c_1((-\Delta_N)u,u)_{\Hzero}&=&c_1\|(-\Delta_N)^{\frac12}u\|_{\Hzero(\Omega)}^2 \leq \|A^{\frac14} u\|_{\Hzero}=(A^{\frac12}u,u)_{\Hzero}\\
&\leq& c_2\|(-\Delta_N)^{\frac12}u\|_{\Hzero(\Omega)}^2 =  c_2 ((-\Delta_N)u,u)_{\Hzero}
\end{eqnarray*}
because of \cite[Chapter I, Corollary~7.1]{Krein}. Thus $c_1(-\Delta_N)\leq A^{\frac12}\leq c_2(-\Delta_N)$. \cite[Chapter I, Corollary~7.1]{Krein} also implies $\mathcal{D}(A^\frac12)= H^1_{(0)}(\Omega)=\mathcal{D}(-\Delta_N)$.
Moreover, we have that there are $c_3,c_4>0$ such that 
\begin{eqnarray*}
  c_3(-\Delta_N u,u)_{\Hzero} = c_3(u,u)_{L^2}\leq (a(c_0)^{-1} u,u)_{L^2}= (Bu,u)_{\Hzero}\leq c_4(-\Delta_N u,u)_{\Hzero} 
\end{eqnarray*}
for all $u\in \Hone(\Omega)$.
Combining this with the previous estimates, we obtain $\rho_1 A^{\frac12}\leq B\leq  \rho_2  A^{\frac12}$ for some $0<\rho_1\leq \rho_2<\infty$. 

Hence we have proved {\bf H1-H2} and conclude that $-A_1|_{H_0}$ generates a bounded analytic semigroup on $H_0$. Moreover, $H_1=H_0\oplus \{(0,m): m\in \R\}$ and $A_1(I-P_0)=(I-P_0)A_1=0$.
Therefore there is some $\delta\in (\frac{\pi}2,\pi)$ such that $\lambda+A_1 \colon \mathcal{D}(A_1)\to H_1$ is invertible  for all $\lambda \in\Sigma_\delta$. Moreover, the resolvent estimate $\|(\lambda+A_1)^{-1}\|_{\mathcal{L}(H)}\leq \frac{C}{|\lambda|}$ for all $\lambda \in\Sigma_\delta$ follows from the corresponding estimate for $A_1|_{H_0}$. Therefore $-A_1$ generates a bounded analytic semigroup on $H_1$.

The equivalence of norms follows from the invertibility of $A_1|_{H_0}$ and the boundedness of $A_1\colon (H^3(\Omega)\cap H^2_N(\Omega))\times \Hone(\Omega)\to H_1$.
\end{proof}
\begin{rem}\label{rem:Spectrum}
As shown above there is some $\delta>\frac{\pi}2$ such that $\Sigma_\delta \subseteq \rho(-A_1)$. In the special case that $\tilde{\nu}\equiv \nu_0, \tilde{\eta}\equiv \eta_0$, and $\rho_0 \equiv 1$ are constant, we have $a(c_0)\equiv 2\nu_0+\eta_0$ and $\sigma(-A_1)$ consists of the eigenvalues 
\begin{equation*}
  \lambda_n^\pm= \sqrt{\frac{\eps}{\alpha\beta}\mu_n}e^{\pm i\theta}, \quad e^{i\theta}=-\kappa +i\sqrt{1-\kappa^2},\quad \kappa = \frac{\alpha\beta^2 (2\nu_0+\mu_0)}\eps
\end{equation*}
provided that $0<\kappa<1$,
where $\mu_n$ are the eigenvalues of $\Delta_N\Delta$, cf. \cite[Lemma~A.1]{ChenTriggiani}.
Note that $\frac{\pi}2<\theta <\pi$ and $\theta \to \frac{\pi}2$ as $\kappa\to 0$. Hence $\delta>\frac{\pi}2$ above can be arbitrarily close to $\frac{\pi}2$ in certain situations.
\end{rem}

Because of the triangular structure of $A$, we conclude from the latter lemma:
\begin{prop}
  $-\A$ generates a bounded analytic semigroup on $H$. Moreover, $\|\A u\|_H+\|u\|_H$ is equivalent to $\|\vc{w}\|_{H^2}+ \|g\|_{H^1}+ \|c\|_{H^3}$, where $u=(c',g,\vc{w})^T$.
\end{prop}
\begin{proof}
  Let $A_\gamma\ve\equiv A_\gamma(c_0) \ve= - P_\sigma \Div \widetilde{\tn{S}}(c_0,\tn{D}\ve)=-P_\sigma\Div (\tilde{\nu}(c_0)\tn{D}\ve)$ for all $\ve\in       \mathcal{D}(A_\gamma)   = \left\{\ue\in H^2(\Omega)^d\cap L^2_\sigma(\Omega): T_\gamma \ue =0\right\}$. Then $-A_\gamma(c_0) \colon \mathcal{D}(A_\gamma)\subset L^2_\sigma(\Omega) \to L^2_\sigma(\Omega)$  generates a bounded analytic $C^0$-semigroup because of Theorem~\ref{thm:Positivity} below. 
  Moreover, $A_\gamma$ is invertible and  $\|A \ve\|_{L^2_\sigma(\Omega)}$ is equivalent to  $\|\ve\|_{H^2(\Omega)}$ because of \cite[Lemma~4]{ModelH}. Hence there is some $\delta>\frac{\pi}2$ such that $(\lambda +A)^{-1}$ exists for all $\lambda \in \Sigma_\delta$  and
  \begin{equation*}
    (\lambda +A)^{-1} = \left(
      \begin{array}{cc}
        (\lambda +A_1)^{-1} & -(\lambda +A_1)^{-1}A_2(\lambda +A_\gamma)^{-1} \\
        0 & (\lambda +A_\gamma)^{-1}
      \end{array}
\right).
  \end{equation*}
  Using 
  \begin{eqnarray*}
    |\lambda|\|(\lambda +A_\gamma)^{-1}\|_{\mathcal{L}(L^2_\sigma)}+ \|A_\gamma(\lambda +A_\gamma)^{-1}\|_{\mathcal{L}(L^2_\sigma)}\leq C_\delta\quad \text{for all}\ \lambda\in\Sigma_\delta
  \end{eqnarray*}
  we conclude that
  \begin{eqnarray*}
    \|A_2(\lambda +A_\gamma)^{-1}\vc{f}_2\|_{H_1} &\leq& C\|(\lambda +A_\gamma)^{-1}\vc{f}_2\|_{H^2(\Omega)}\\
&\leq& C'\|A_\gamma(\lambda +A_\gamma)^{-1}\vc{f}_2\|_{L^2_\sigma(\Omega)} \leq C''\|\vc{f}_2\|_{L^2(\Omega)}
  \end{eqnarray*}
  for all $\vc{f}_2\in L^2_\sigma(\Omega)$, $\lambda\in\Sigma_\delta$.
  Therefore we easily obtain
  \begin{equation*}
    \|(\lambda +\A)^{-1}\|_{H}\leq \frac{C}{|\lambda|} \qquad \text{uniformly in}\ \lambda\in \Sigma_\delta.
  \end{equation*}
  Hence $-\A$ generates a bounded analytic $C^0$-semigroup. Finally, the equivalence of norms can be easily shown using the resolvent identity above for $\lambda=1$ and the corresponding statement in Lemma~\ref{lem:A1}.
\end{proof}
\begin{cor}\label{cor:Generation}
  Let $c_0\in H^2(\Omega)$. Then
  $\A+\B$ generates an analytic $C^0$-semigroup. Moreover, $\|(\A+\B) u\|_{H}+ \|u\|_{H}$ is equivalent to $\|\we\|_{H^2}+ \|g\|_{H^1}+ \|c\|_{H^3}$, where $u=(c,g,\we)^T\in \mathcal{D}(\A)$.
\end{cor}
\begin{proof}
  The corollary follows from the fact that for every $\eps>0$ there is a constant $C_\eps>0$ such that
  \begin{equation*}
    \|\B u\|_{H} \leq \eps \|\A u\|_{H} + C_\eps \|u\|_{H}
  \end{equation*}
  for all $u\in \mathcal{D}(A)$ and a standard result from semigroup theory, cf. e.g. \cite[Chapter 3, Theorem 2.1]{Pazy}. 
  The latter estimate follows from
  \begin{eqnarray*} 
    \|\B u\|_{H} &\leq& C(c_0)\|g\|_{H^\frac34(\Omega)}
    \leq  C(c_0)\|g\|_{L^2(\Omega)}^\frac14\|g\|_{H^1(\Omega)}^{\frac34} 
  \leq  C(c_0)\|g\|_{H^{-1}_{(0)}(\Omega)}^\frac18\|g\|_{H^1(\Omega)}^{\frac78} \\
    &\leq& C(c_0)\|\A u\|_{H}^\frac78\|u\|_{H}^\frac18 \leq \eps\|\A u\|_{H} + C_\eps(c_0) \|u\|_H
  \end{eqnarray*}
for every $\eps>0$,
  where $u=(c,g,\we)$ and we have used Lemma~\ref{lem:LowerOrder}, $H^{\frac34}(\Omega)=(L^2(\Omega),H^1(\Omega))_{\frac34,2}$, 
  and \eqref{eq:InterpolationInequality}.
\end{proof}
\begin{lem} Let $\A, \mathcal{D}(\A), H$ be defined as above. Then
  \begin{equation*}
    (\mathcal{D}(\A),H)_{\frac12,2} =  H^2_N(\Omega)\times L^2_{(0)}(\Omega)\times(H^1(\Omega)^d\cap L^2_\sigma(\Omega)).
  \end{equation*}
\end{lem}
\begin{proof}
  We only need to show that
  \begin{eqnarray*}
    (H^1_{(0)}(\Omega),\Hzero(\Omega))_{\frac12,2} &=& L^2_{(0)}(\Omega),\\
    (H^3(\Omega)\cap H^2_N(\Omega),H^1(\Omega))_{\frac12,2} &=& H^2_N(\Omega),\\
    (\mathcal{D}(A_\gamma),L^2_\sigma(\Omega))_{\frac12,2} &=& H^1(\Omega)^d\cap L^2_\sigma(\Omega),
  \end{eqnarray*}
where $\mathcal{D}(A_\gamma)=\{\ue \in H^2(\Omega)^d\cap L^2_\sigma(\Omega):T_\gamma \ue =0\}$.
  The first equality follows from $\Hone(\Omega)=P_0H^1(\Omega), L^2_{(0)}(\Omega)= P_0 L^2(\Omega), \Hzero(\Omega) = P_0 H^{-1}(\Omega)$ and \cite[Section 1.2.4, Theorem]{Triebel1}. The second identity is proved using that $H^3(\Omega)\cap H^2_N(\Omega)=(1-\Delta_N)^{-1} H^1(\Omega)$, $H^1(\Omega)= (1-\Delta_N)^{-1} H^{-1}(\Omega)$, and $H^2_N(\Omega)= (1-\Delta_N)^{-1}L^2(\Omega)$. The third identity follows from Lemma~\ref{lem:Interpol}, below.
\end{proof}

\noindent
\begin{proof*}{of Theorem~\ref{thm:LinearPart}}
As seen at the beginning of Section~\ref{sec:LinearPart} we can assume without loss of generality that $\vc{a}\equiv 0$.
Let $(\vc{f}_1,f_2,0, \ve_0,c_0')\in Y_T$ and $f\in L^2(0,T;H)$ be defined as in \eqref{eq:AbstractEvo} and extend $f$ by zero for $t\geq T$. Moreover, let $g_0=\Div \ve_0$, $\vc{w}_0=P_\sigma \ve_0$, and let $u_0$ be as in \eqref{eq:AbstractEvo2}.

  Applying Theorem~\ref{thm:MaxReg}  there is a unique solution ${u}$ of
  \begin{eqnarray*}
    \frac{du}{dt}(t) + \A u(t) + \mathcal{B}u(t) &=& f(t),\qquad t>0,\\
    u(0)&=& u_0 
  \end{eqnarray*}
  and
  \begin{equation*}
    \|(\partial_t {u}, (\A+\B){u})\|_{L^2(0,\infty;H)} \leq C \left(\|f\|_{L^2(0,\infty;H)} +\|u_0\|_{(\mathcal{D}(\A),H)_{\frac12,2}}\right)
  \end{equation*} 
  Using that $\|{u}\|_{L^2(0,T_0;H)}\leq C(T_0)\left( \|\partial_t {u}\|_{L^2(0,T_0;H)}+\|u_0\|_{H} \right)$ for any fixed $0<T_0<\infty$, we obtain that $u$ restricted to $(0,T)$ satisfies
  \begin{equation*}
    \|({u},\partial_t {u}, (\A+\B){u})\|_{L^2(0,T;H)} \leq C(T_0) \left(\|f\|_{L^2(0,T;H)} +\|u_0\|_{(\mathcal{D}(\A),H)_{\frac12,2}}\right)
  \end{equation*} 
  uniformly in $0<T\leq T_0$ and $(f,u_0)$. 
  Hence $u=(c',g,\we)$ solves \eqref{eq:lin2'}-\eqref{eq:Split2} with $G(\Div \ve)=\Delta_N^{-1}g$. 
Therefore $(\ve,c')$ with $\ve=\we+ \nabla \Delta_N^{-1} g$ solves \eqref{eq:lin1'}-\eqref{eq:lin4'}, which implies 
\begin{equation*}
L(c_0)
\begin{pmatrix}
  \ve\\ c'
\end{pmatrix}
=
\begin{pmatrix}
  \vc{f}_1\\ 0\\ f_2
\end{pmatrix}.
\end{equation*}
 The estimate of $(\ve,c')\in X_T$, the continuity of $L(c_0)^{-1}\colon Y_T\to X_T$ follows from the estimate above and the equivalence of norms stated in Corollary~\ref{cor:Generation}.
\end{proof*}

\appendix
\section{Stokes Operator with Navier Boundary Conditions}

In this appendix we summarize some results for the Stokes operator with variable viscosity in the case of Navier boundary conditions. More detailed information can be found in \cite[Chapter~5]{Habilitation}.

We consider
 \begin{equation*}
   A_\gamma(c) \colon  \mathcal{D}(A_\gamma(c))\subset L^2_\sigma(\Omega)\to L^2_\sigma(\Omega)\colon \ve\mapsto A_\gamma(c)\ve:=-P_2\Div (2\nu(c)D\ve),  
 \end{equation*}
 where $c\in W^1_q(\Omega)$, $q>d$, and
 \begin{equation*}
      \mathcal{D}(A_\gamma(c))   = \left\{\ue\in H^2(\Omega)^d\cap L^2_\sigma(\Omega): (2\vc{n}\cdot\nu(c)\tn{D} \ue)_\tau + \gamma(c) \ue_\tau =0\right\}.
 \end{equation*}
 Here we assume that $c\in W^1_q(\Omega)$ for some $q>d$, $\gamma\in C^1(\R)$ with $0\leq \gamma(s) <\infty$ for all $s\in\R$, and $\Omega\subset \R^d$, $d=2,3$, is a bounded domains with $C^3$-boundary. -- We note that $H^2(\Omega)\hookrightarrow W^1_q(\Omega)$ for some $q>d$ if $d=2,3$. -- If $\inf_{s\in\R}\gamma(s)=0$ we assume additionally that $\Omega$ does not have an axis of symmetry, i.e., $R= \{0\}$, where
\begin{alignat}{2}\label{eq:R1}
  R&= \{\ve\in H^1_n(\Omega): \ve(x)= \vc{a}+ \vc{b}\wedge x, \vc{a},\vc{b}\in \R^3  \}&\quad& \text{if}\ d=3,\\
R&=\left\{\ve\in H^1_n(\Omega): \ve(x)= \vc{a}+b 
\begin{pmatrix}
  -x_2\\ x_1
\end{pmatrix}, \vc{a}\in \R^2, b\in \R
\right\} &\quad& \text{if}\ d=2.
\end{alignat}
In this case we have the Korn inequality
\begin{equation}
  \label{eq:SecondKorn}
  \|\ve\|_{H^1} \leq C\|\tn{D}\ve\|_{L^2} \qquad \text{for all} \ \ve\in H^1_n(\Omega),
\end{equation}
cf. Ne\v{c}as~\cite[Theorem~3.5]{Necas} for the case $d=3$. If $d=2$, the inequality follows from the three dimensional estimate by  extending $\ve\in H^1_n(\Omega)^2$ to $\tilde{\ve}(x_1,x_2,x_3)= 
(
  \ve_1(x_1,x_2), \ve_2(x_1,x_2), 0
)^T
\in H^1_n(\Omega\times (-1,1))$.

Because of 
\begin{eqnarray}\label{eq:WeakBCNS}
  -(\Div (2\nu(c) \tn{D}\ve), \we)_{L^2(\Omega)} 
&=& (2\nu(c)\tn{D}\ve,\tn{D}\we)_{L^2(\Omega)} +  (\gamma(c)\ve,\we)_{L^2(\partial\Omega)}\\\nonumber
&=&
  -(\ve, \Div (2\nu(c) \tn{D}\we)_{L^2(\Omega)}  
\end{eqnarray}
for all $\ve,\we\in \mathcal{D}(A_\gamma(c))$, $A_\gamma(c)$ is a symmetric operator.
Moreover, if $\inf_{s\in\R} \gamma(s)=0$, then $\nu(s)\geq \nu_0>0$ and \eqref{eq:SecondKorn} implies
\begin{equation*}
  -(\Div (2\nu(c) \tn{D}\ve), \ve)_{L^2(\Omega)}\geq c_0\|\ve\|_{H^1(\Omega)}^2 \qquad \text{for all}\ \ve \in \mathcal{D}(A_\gamma)
\end{equation*}
for some $c_0>0$. If $ \gamma_0:=\inf_{s\in\R} \gamma(s)>0$, then one obtains  
\begin{equation*}
  -(\Div (2\nu(c) \tn{D}\ve), \ve)_{L^2(\Omega)}\geq C(\gamma_0)\left(\|\tn{D}\ve\|_{L^2(\Omega)}^2 + \|\ve\|_{L^2(\partial\Omega)}^2\right)\geq C'(\gamma_0)\|\ve\|_{H^1(\Omega)}
\end{equation*}
because of $\|\we\|_{H^1(\Omega)}\leq C\|\tn{D}\we\|_{L^2(\Omega)}$ for any $\we\in H^1_0(\Omega)^d$.
Furthermore, we have:
 \begin{thm}\label{thm:Positivity}
   Let $c\in W^{1}_q(\Omega)$, $q>d$. Then $A_\gamma(c)$ is a positive self-adjoint operator on $L^2_\sigma(\Omega)$.
 \end{thm}
 \begin{proof}
First of all, $A_\gamma(c)\colon \mathcal{D}(A_\gamma)\to L^2_\sigma(\Omega)$ is invertible because of the following arguments: By the Lemma of Lax-Milgram for every $\vc{f}\in L^2_\sigma(\Omega)$ there is a unique $\ve \in H^1_n(\Omega)\cap L^2_\sigma(\Omega)$ such that
\begin{equation*}
  (2\nu(c)\tn{D}\ve,\tn{D}\we)_{L^2(\Omega)} +  (\gamma(c)\ve,\we)_{L^2(\partial\Omega)} = (\vc{f},\we)_{L^2(\Omega)}\qquad \text{for all}\ \we \in \mathcal{D}(A_\gamma).
\end{equation*}
Since $\gamma(c) \ve|_{\partial\Omega} \in H^{\frac12}(\partial\Omega)^d$, there is some $\ue\in H^2(\Omega)^d\cap H^1_0(\Omega)^d\cap L^2_\sigma(\Omega)$ such that $(\vc{n}\cdot 2\nu(c)\tn{D} \ue)_\tau= -\gamma(c) \ve|_{\partial\Omega}$ due to Lemma~\ref{lem:Extension}. By Gau\ss' theorem we obtain that $\tilde{\ve}= \ve+\ue$ solves
\begin{equation*}
  (2\nu(c)\tn{D}\tilde{\ve},\tn{D}\we)_{L^2(\Omega)}  = (\vc{f}-\Div (2\nu \tn{D}\ue),\we)_{L^2(\Omega)}\qquad \text{for all}\ \we \in \mathcal{D}(A_\gamma).
\end{equation*}
Because of \cite[Theorem~5.2.3]{Habilitation}, we conclude that $\tilde{\ve}\in H^2(\Omega)^d$, which yields $\ve\in H^2(\Omega)^d$. 
 Since $A_\gamma(c)$ is symmetric due to \eqref{eq:WeakBCNS}, $A_\gamma(c)$ is self-adjoint.
\end{proof}

Finally, we need:
\begin{lem}\label{lem:Interpol}
 Let $\gamma\colon \R\to [0,\infty)$ be continuously differentiable, $c\in W^1_q(\Omega)$ for some $q>d$, and assume that $\Omega$ possesses no axis of symmetry if $\inf_{s\in\R}\gamma(s)=0$. 
 Then
\begin{equation*}
  (L^2_\sigma(\Omega),\mathcal{D}(A_\gamma))_{\frac12,2} = H^1(\Omega)\cap L^2_\sigma(\Omega).
\end{equation*}
\end{lem}
\begin{proof}
We use that $A_\gamma(c)$ is an invertible, self-adjoint and positive operator on $L^2_\sigma(\Omega)$. Hence we can use Theorem~\ref{thm:MaxReg}  to conclude that for every $\ue_0\in H^1(\Omega)\cap L^2_\sigma(\Omega)$ there is some $\ue\in H^1(0,\infty;L^2_\sigma(\Omega))\cap L^2(0,\infty; \mathcal{D}(A_\gamma)$ such that $\ue|_{t=0}=\ue_0$. Thus 
\begin{equation*}
  (L^2_\sigma(\Omega),\mathcal{D}(A_\gamma))_{\frac12,2} \supseteq H^1(\Omega)^d\cap L^2_\sigma(\Omega)
\end{equation*}
by the trace method of real interpolation.
But the converse inclusion holds since for every $\ue\in H^1(0,\infty;L^2_\sigma(\Omega))\cap L^2(0,\infty; \mathcal{D}(A_\gamma))$ we obviously have $\ue_0=\ue|_{t=0}\in H^1(\Omega)^d\cap L^2_\sigma(\Omega)$ because of $(L^2(\Omega),H^2(\Omega))_{\frac12,2}= H^1(\Omega)$ and (\ref{eq:BUCEmbedding}).
\end{proof}

\def\cprime{$'$}


\end{document}